\documentclass[reqno,centertags, 12pt]{amsart}
\usepackage{amsmath,amsthm,amscd,amssymb}
\usepackage{latexsym}
\usepackage{graphicx}
\usepackage[mathscr]{eucal}
\newcommand{\bbC}{{\mathbb{C}}}
\newcommand{\bbD}{{\mathbb{D}}}

\newcommand{\bbG}{{\mathbb{G}}}

\newcommand{\bbN}{{\mathbb{N}}}

\newcommand{\bbR}{{\mathbb{R}}}

\newcommand{\bbZ}{{\mathbb{Z}}}

\newcommand{\fre}{{\frak{e}}}

\newcommand{\x}{{\mathbf{x}}}
\newcommand{\z}{{\mathbf{z}}}

\newcommand{\calE}{{\mathcal{E}}}
\newcommand{\calF}{{\mathcal F}}

\newcommand{\calL}{{\mathcal L}}

\newcommand{\calS}{{\mathcal S}}
\newcommand{\calT}{{\mathcal T}}

\newcommand{\bddot}{{\boldsymbol{\cdot}}}
\newcommand{\dott}{\,\cdot\,}
\newcommand{\no}{\nonumber}
\newcommand{\lb}{\label}
\newcommand{\f}{\frac}
\newcommand{\ul}{\underline}
\newcommand{\ol}{\overline}
\newcommand{\ti}{\tilde  }
\newcommand{\wti}{\widetilde  }

\newcommand{\Wr}{\text{\rm{Wr}}}

\newcommand{\dist}{\text{\rm{dist}}}

\newcommand{\Sz}{\text{\rm{Sz}}}

\newcommand{\spann}{\text{\rm{span}}}
\newcommand{\rank}{\text{\rm{rank}}}
\newcommand{\ran}{\text{\rm{ran}}}

\newcommand{\ess}{\text{\rm{ess}}}

\newcommand{\s}{\text{\rm{s}}}

\newcommand{\intt}{\text{\rm{int}}}

\newcommand{\bi}{\bibitem}

\newcommand{\beq}{\begin{equation}}
\newcommand{\eeq}{\end{equation}}
\newcommand{\ba}{\begin{align}}
\newcommand{\ea}{\end{align}}
\newcommand{\veps}{\varepsilon}

\newcommand{\vw}{{\vec{w\!}}\,}

\DeclareMathOperator{\ca}{cap}
\newcounter{smalllist}
\newenvironment{SL}{\begin{list}{{\rm\roman{smalllist})}}{%
\setlength{\topsep}{0mm}\setlength{\parsep}{0mm}\setlength{\itemsep}{0mm}%
\setlength{\labelwidth}{2em}\setlength{\leftmargin}{2em}\usecounter{smalllist}%
}}{\end{list}}

\newcommand{\comm}[1]{}
\DeclareMathOperator{\Real}{Re}
\DeclareMathOperator{\Ima}{Im}

\DeclareMathOperator*{\slim}{s-lim}

\allowdisplaybreaks
\numberwithin{equation}{section}
\newtheorem{theorem}{Theorem}[section]

\newtheorem*{p2.1}{Proposition 2.1}
\newtheorem{proposition}[theorem]{Proposition}
\newtheorem{lemma}[theorem]{Lemma}
\newtheorem{corollary}[theorem]{Corollary}
\theoremstyle{definition}
\newtheorem{oq}{Open Question}

\newtheorem*{remark}{Remark}
\newtheorem*{remarks}{Remarks}
\newtheorem*{definition}{Definition}

\newcommand{\abs}[1]{\lvert#1\rvert}
\newcommand{\jap}[1]{\langle #1 \rangle}
\newcommand{\norm}[1]{\lVert#1\rVert}

\begin{document}

\title[Finite Gap Jacobi Matrices, II]{Finite Gap Jacobi Matrices,\\II. The Szeg\H{o} Class}
\author[J.~S.~Christiansen, B.~Simon, and M.~Zinchenko]{Jacob S.~Christiansen$^{1}$, Barry Simon$^{2,3}$,\\and
Maxim Zinchenko$^2$}

\thanks{$^1$ Department of Mathematical Sciences, University of Copenhagen,
Universitetsparken 5, DK-2100 Copenhagen, Denmark. E-mail: stordal@math.ku.dk.
Supported in part by a Steno Research Grant from FNU, the Danish Research Council}

\thanks{$^2$ Mathematics 253-37, California Institute of Technology, Pasadena, CA 91125.
E-mail: bsimon@caltech.edu; maxim@caltech.edu}

\thanks{$^3$ Supported in part by NSF grant DMS-0652919}

%\date{May 31, 2009}    %Submitted
%\date{October 9, 2009} %Updated
\date{November 16, 2009} %Add Yuditskii

\begin{abstract} Let $\fre\subset\bbR$ be a finite union of disjoint closed intervals. We study
measures whose essential support is $\fre$ and whose discrete eigenvalues obey a $1/2$-power condition.
We show that a Szeg\H{o} condition is equivalent to
\[
\limsup \f{a_1\cdots a_n}{\ca(\fre)^n}>0
\]
(this includes prior results of Widom and Peherstorfer--Yuditskii). Using Remling's extension of the
Denisov--Rakhmanov theorem and an analysis of Jost functions, we provide a new proof of Szeg\H{o} asymptotics,
including $L^2$ asymptotics on the spectrum. We use heavily the covering map formalism of Sodin--Yuditskii
as presented in our first paper in this series.
\end{abstract}

\maketitle

%%%%%%%%%%%%%%%%%%%%%%%%%%%%%%%%%%%%%%%%%%%%%%%%%%%%%%%%%%%%%%
\section{Introduction} \lb{s1}
%%%%%%%%%%%%%%%%%%%%%%%%%%%%%%%%%%%%%%%%%%%%%%%%%%%%%%%%%%%%%%

In this paper, we study Jacobi matrices, $J$, and asymptotics of the associated orthogonal polynomials (OPRL), where
$\sigma_\ess(J)$ is a finite gap set, $\fre$. By this we mean that $\fre$ is a finite union of disjoint closed intervals,
\begin{equation} \lb{1.1}
\fre =\bigcup_{j=1}^{\ell+1} \, [\alpha_j,\beta_j] \qquad
\alpha_1 < \beta_1 < \alpha_2 < \cdots < \beta_{\ell+1}
\end{equation}
$\ell$ counts the number of gaps, that is, bounded open intervals in $\bbR\setminus\fre$.

We recall that a Jacobi matrix is a tridiagonal matrix which we label
\begin{equation} \lb{1.1a}
J=
\begin{pmatrix}
b_1 & a_1 & 0  & \cdots \\
a_1 & b_2 & a_2  & \cdots \\
0 & a_2 & b_3 & \cdots \\
\vdots & \vdots & \vdots & \ddots
\end{pmatrix}
\end{equation}
The Jacobi parameters $\{a_n,b_n\}_{n=1}^\infty$ have $a_n >0$ and $b_n\in\bbR$. There is a one-one
correspondence between probability measures, $d\mu$, of compact support on $\bbR$ and bounded Jacobi matrices
where $d\mu$ is the spectral measure for $J$ and the vector $(1,0,\dots)^t$. Moreover, $d\mu$ determines $J$
via recursion relations for the orthonormal polynomials, $p_n(x)$, which are ($a_0\equiv 0$)
\begin{equation} \lb{1.1b}
xp_n(x) =a_{n+1} p_{n+1}(x) + b_{n+1} p_n(x) + a_n p_{n-1}(x)
\end{equation}
See \cite{SzBk,GBk,OPUC1,Rice} for background on OPRL.

This paper is the second in a series---the first, \cite{CSZ1}, henceforth called paper~I, studied the isospectral
torus, an $\ell$-dimensional family of two-sided almost periodic Jacobi matrices with essential spectrum, $\fre$,
about which we'll say more later in this introduction. We note for now that these matrices have periodic
coefficients if and only if the harmonic measure of the intervals $[\alpha_j,\beta_j]$ are all rational (i.e., if
$d\rho_\fre$ is the potential theoretic equilibrium measure for $\fre$, then each $\rho_\fre ([\alpha_j,\beta_j])$
is rational; for background on potential theory in spectral analysis, see \cite{StT,EqMC}). We'll call this
the periodic case.

In the current paper, we want to study Szeg\H{o}'s theorem for the general finite gap case. Of course, the
phrase ``Szeg\H{o}'s theorem'' can be ambiguous since Szeg\H{o} was so prolific, but by this we mean a set of
results concerned with leading asymptotics in the theory of orthogonal polynomials on the unit circle (OPUC).
Even here, there is ambiguity since some of the results can be interpreted in terms of Toeplitz determinants and
there are several related objects. Indeed, we'll distinguish between what we call Szeg\H{o}'s theorem and Szeg\H{o}
asymptotics.

In the OPUC case, the recursion parameters $\{\alpha_n\}_{n=0}^\infty$ lie in $\bbD=\{z\mid\abs{z}<1\}$ and are
called Verblunsky coefficients. We use $\varphi_n(z)$ for the orthonormal polynomials and write the measure $d\mu$ as
\begin{equation} \lb{1.2}
d\mu(\theta) = w(\theta)\, \f{d\theta}{2\pi} + d\mu_\s(\theta)
\end{equation}
where $d\mu_\s$ is $d\theta/2\pi$-singular. One also defines $\rho_n =(1-\abs{\alpha_n}^2)^{1/2}$ (see
\cite{SzBk, GBk,OPUC1,OPUC2,1Ft} for background on OPUC).

Then what we'll call Szeg\H{o}'s theorem for OPUC says that
\begin{equation} \lb{1.3}
\lim_{N\to\infty} \prod_{n=0}^N \rho_n = \exp\biggl( \int_0^{2\pi} \log (w(\theta)) \f{d\theta}{2\pi}\biggr)
\end{equation}
Notice that since $\rho_n \leq 1$, the limit on the left always exists, although it may be $0$. By Jensen's inequality,
the integral on the right is nonpositive, but may diverge to $-\infty$, in which case we interpret the exponential as
$0$. It is easy to see that the left side is nonzero if and only if $\sum_{n=0}^\infty \abs{\alpha_n}^2 <\infty$. Thus,
\eqref{1.3} implies
\begin{equation} \lb{1.4}
\sum_{n=0}^\infty \, \abs{\alpha_n}^2 <\infty \quad \Leftrightarrow \quad
\int \log (w(\theta))\, \f{d\theta}{2\pi} >-\infty
\end{equation}

By Szeg\H{o} asymptotics, we mean the fact that when both conditions in \eqref{1.4} hold, there is an explicit
nonvanishing function, $G$, on $\bbC\setminus\ol{\bbD}$ so that for $z$ in that set,
\begin{equation} \lb{1.5}
\lim_{n\to\infty} z^{-n} \varphi_n(z) = G(z)
\end{equation}
In terms of the conventional Szeg\H{o} function,
\begin{equation}
D(z)=\exp \biggl(\int \f{e^{i\theta}+z}{e^{i\theta}-z} \log(w(\theta)) \f{d\theta}{2\pi}\biggr), \quad z\in\bbD
\end{equation}
we have $G(z)=\ol{D(1/\bar z)}^{-1}$.

Analogs of Szeg\H{o}'s theorem for OPRL, where $\fre$ is a single interval (typically $\fre=[-1,1]$ or $[-2,2]$),
were found initially by Szeg\H{o} \cite{Sz22a}, with important developments by Shohat \cite{Sh} and Nevai
\cite{Nev79}. These works suppose no or finitely many eigenvalues outside $\fre$. The natural condition on
eigenvalues (see \eqref{1.7} and \eqref{1.9} below) was found by Killip--Simon \cite{KS} and Peherstorfer--Yuditskii
\cite{PYpams}. The best form of Szeg\H{o}'s theorem (with a Szeg\H{o} condition; see below) is

\begin{theorem}[Simon--Zlato\v{s}] \lb{T1.1} Let $J$ be a Jacobi matrix with essential spectrum $[-2,2]$,
$\{a_n,b_n\}_{n=1}^\infty$ its Jacobi parameters, $\{x_k\}$ a listing of its eigenvalues outside $[-2,2]$,
and
\begin{equation} \lb{1.6}
d\mu(x) =w(x)\, dx + d\mu_\s(x)
\end{equation}
its spectral measure. Define
\begin{equation} \lb{1.7}
\calE(J) =\sum_k \, (\abs{x_k}-2)^{1/2}
\end{equation}
and
\begin{equation} \lb{1.8}
A_n = a_1 \cdots a_n \qquad \bar A=\limsup A_n \qquad \ul{A\!} = \liminf A_n
\end{equation}
Consider the three conditions:
\begin{SL}
\item[{\rm{(i)}}] Szeg\H{o} condition
\begin{equation} \lb{1.8x}
\int_{-2}^2 \log(w(x)) (4-\abs{x}^2)^{-1/2} \, dx >-\infty
\end{equation}

\item[{\rm{(ii)}}] Blaschke condition
\begin{equation} \lb{1.9}
\calE(J) <\infty
\end{equation}

\item[{\rm{(iii)}}] Widom condition
\begin{equation} \lb{1.10}
0 < \ul{A\!} \, \leq \bar A <\infty
\end{equation}
\end{SL}
Then any two of {\rm{(i)--(iii)}} imply the third, and if they hold, the following have limits as $N\to\infty$:
\begin{equation} \lb{1.11}
A_N, \quad \sum_{n=1}^N b_n, \quad \sum_{n=1}^N (a_n-1)
\end{equation}
and
\begin{equation} \lb{1.12}
\sum_{n=1}^\infty\, \abs{a_n-1}^2 + \abs{b_n}^2 <\infty
\end{equation}
\end{theorem}

Before leaving our summary of the case $\fre=[-2,2]$, we note that Damanik--Simon \cite{Jost1} have proven
Szeg\H{o} asymptotics in some cases where the Szeg\H{o} condition fails. This will not concern us here, but will
be studied in the finite gap case in paper~III \cite{CSZ3}.

In Section~\ref{s4}, we prove a precise analog of the statement ``any two of (i)--(iii) imply the third'' for general
finite gap sets, $\fre$. We note that for the periodic case, this is a prior result of Damanik--Killip--Simon
\cite{DKS}. There are also prior results for the general finite gap case in Widom \cite{Widom}, Aptekarev
\cite{Apt}, and Peherstorfer--Yuditskii \cite{PY,PYarx}; see Section~\ref{s4} for more details.

The limit results, \eqref{1.11} and \eqref{1.12}, need modification, however. First, even in the general
one-interval case, one needs $a_1 \cdots a_n/C^n$ for a suitable constant $C$. The theory of regular measures
\cite{StT,EqMC} says the right value of $C$ must be $\ca(\fre)$, the logarithmic capacity of $\fre$---a result
that, in this context, goes back at least to Widom \cite{Widom} who also discovered that $a_1 \cdots a_n/
\ca(\fre)^n$ doesn't have a limit but is only asymptotically almost periodic.

These limit results are expressed most naturally in terms of the isospectral torus associated to $\fre$. For any
Jacobi matrix obeying the analogs of (i)--(iii), there is an element $\{\ti a_n, \ti b_n\}_{n=1}^\infty$ of
the isospectral torus so that
\begin{equation} \lb{1.13}
\lim_{n\to\infty} \abs{a_n-\ti a_n} + \abs{b_n - \ti b_n} =0
\end{equation}
This result, which goes back to Aptekarev \cite{Apt} and Peherstorfer--Yuditskii \cite{PY,PYarx} using
variational methods, will be proven with our techniques in Section~\ref{s6}, where we'll also prove that
$\lim (a_1 \cdots a_n/\ti a_1 \cdots \ti a_n)$ exists and is nonzero. (In paper~I, we proved that in the
isospectral torus, $\ti a_1 \cdots \ti a_n/\ca(\fre)^n$ is almost periodic in $n$.)

An interesting open question concerns the analog of \eqref{1.12}:

\begin{oq} \lb{OQ1.1} Is $\sum_{n=1}^\infty \abs{a_n-\ti a_n}^2 + \abs{b_n-\ti b_n}^2 <\infty$ when the
analogs of (i)--(iii) hold?
\end{oq}

In Section~\ref{s7}, we'll prove an analog of Szeg\H{o} asymptotics, namely, away from the interval $[\alpha_1,
\beta_{\ell+1}]$, the ratio $p_n(z)/\ti p_n(z)$ has a nonzero limit where $\ti p_n$ are the OPRL for $\{\ti a_n,
\ti b_n\}_{n=1}^\infty$.

Let us next summarize some of the techniques we'll use below, in part to standardize some notation. Coefficient
stripping plays an important role in the analysis: if $J$ has Jacobi parameters $\{a_k, b_k\}_{k=1}^\infty$,
then the $n$-times stripped Jacobi matrix, $J^{(n)}$, is the one with parameters $\{a_{n+k}, b_{n+k}\}_{k=1}^\infty$,
that is,  with
\begin{equation} \lb{1.14}
a_k (J^{(n)}) = a_{k+n}(J) \qquad b_k (J^{(n)}) = b_{n+k}(J)
\end{equation}

If the $m$-function of $J$ is defined on $\bbC_+ = \{z\mid\Ima z>0\}$ by
\begin{equation} \lb{1.15}
m(z,J)=\jap{\delta_1, (J-z)^{-1}\delta_1} = \int \f{d\mu(x)}{x-z}
\end{equation}
then we have the coefficient stripping relation that goes back to Jacobi and Stieltjes,
\begin{equation} \lb{1.16}
m(z,J)^{-1} = -z+b_1 - a_1^2 m(z,J^{(1)})
\end{equation}

We'll make heavy use of the covering space formalism introduced in spectral theory by Sodin--Yuditskii \cite{SY}
and presented with our notation in paper~I. $\x(z)$ is the unique meromorphic map of $\bbD$ to $\bbC\cup\{\infty\}
\setminus\fre$ which is locally one-one with
\begin{equation} \lb{1.17}
\x(z)=\f{x_\infty}{z} + O(1)
\end{equation}
near $z=0$ and $x_\infty >0$.

There is a (Fuchsian) group, $\Gamma$, of M\"obius transformations of $\bbD$ onto itself so that
\begin{equation} \lb{1.18}
\x(z)=\x(w) \quad \Leftrightarrow \quad \exists \gamma\in\Gamma \text{ so that } \gamma(z)=w
\end{equation}
A natural fundamental set, $\calF$, is defined as follows:
\begin{equation} \lb{1.19}
\calF^\intt = \{z\mid\abs{z} < \abs{\gamma(z)}, \, \text{all } \gamma\neq1,\, \gamma\in\Gamma\}
\end{equation}
$\partial\calF^\intt\cap\bbD$ is then $2\ell$ orthocircles, $\ell$ in each half-plane. $\calF$ is $\calF^\intt$
union the $\ell$ orthocircles in $\bbC_+$. $\x$ is then one-one and onto from $\calF$ to $\bbC\cup\{\infty\}
\setminus\fre$.

$\calL$, the set of limit points of $\Gamma$, is defined as $\ol{\{\gamma(0)\mid\gamma\in\Gamma\}}\cap\partial\bbD$.
$\x$ can be meromorphically extended from $\bbD$ to all of $\bbC\cup\{\infty\}\setminus\calL$, or alternatively,
there is a map $\x^\sharp\colon\bbC\cup\{\infty\}\setminus\calL$ to $\calS$, the two-sheeted Riemann surface of
$[\prod_{j=1}^{\ell+1}(z-\alpha_j)(z-\beta_j)]^{1/2}$. All this is described in more detail in paper~I of this series.

That paper also discusses Blaschke products, $B(z,w)$, of the Blaschke factors at $\{\gamma(w)\}_{\gamma\in\Gamma}$.
$B(z) \equiv B(z,0)$ is related to the potential theoretic Green's function, $G_\fre(x)$,
for $\fre$ by
\begin{equation} \lb{1.19x}
\abs{B(z)}=e^{-G_\fre (\x(z))}
\end{equation}
which, in particular, implies that near $z=0$,
\begin{equation} \lb{1.20}
B(z) = \f{\ca(\fre)}{x_\infty}\, z+O(z^2)
\end{equation}

Finally, we use heavily the pullback of $m$ to $\bbD$ via
\begin{equation} \lb{1.21}
M(z)= -m(\x(z))
\end{equation}

We end this introduction with a sketch of the contents of this paper. Our approach to Szeg\H{o}'s theorem is a
synthesis of the covering map method and the approach of Killip--Simon \cite{KS}, Simon--Zlato\v{s} \cite{SZ}, and
Simon \cite{S288} used for $\fre=[-2,2]$. As such, step-by-step sum rules are critical. These are found in
Section~\ref{s2}. One disappointment is that we have thus far not succeeded in finding an analog of what has
come to be called the Killip--Simon theorem (from \cite{KS}). This result gives necessary and sufficient conditions
for the case $\fre=[-2,2]$ that $\sum_{n=1}^\infty (a_n-1)^2 + b_n^2 <\infty$.

\begin{oq} \lb{OQ1.2} Is there a Killip--Simon theorem for the general finite gap Jacobi matrix?
\end{oq}

We note that Damanik--Killip--Simon \cite{DKS} have found an analog for the case where each band has harmonic measure
exactly $(\ell+1)^{-1}$.

Section~\ref{s3} provides a technical interlude on eigenvalue limit theorems needed in the later sections.
Section~\ref{s4} proves a Szeg\H{o}-type theorem for general finite gap $\fre$. Section~\ref{s5} defines Jost functions and
Jost solutions. Section~\ref{s6} proves the existence of the claimed $\{\ti a_n,\ti b_n\}_{n=1}^\infty$ in the
isospectral torus and asymptotics of Jost solutions. Section~\ref{s7} proves asymptotic formulae for the orthogonal
polynomials away from the convex hull of $\fre$ (i.e., the interval $[\alpha_1,\beta_{\ell+1}]$), and Section~\ref{s8}
$L^2$ asymptotics on $\fre$.

The idea that we use in Sections~\ref{s6} and \ref{s7} of first proving Jost asymptotics and using that to get Szeg\H{o}
asymptotics is borrowed from an analog for $\fre=[-2,2]$ of Damanik--Simon \cite{Jost1}. But Section~\ref{s7} has a
simplification of their equivalence argument that is an improvement even for $\fre=[-2,2]$. Most of the results in
Sections~\ref{s6}--\ref{s8} are explicit or implicit in Peherstorfer--Yuditskii \cite{PY,PYarx}. We claim two
novelties here. First, the underlying mechanism of our proof of asymptotics is different from their variational
approach. Instead, we use a recent theorem of Remling \cite{Remppt} about approach to the isospectral torus,
together with an analysis of automorphic characters of Jost functions. Second, by using ideas in a different
paper of Peherstorfer--Yuditskii \cite{PYpams}, we can clarify the $L^2$-convergence result of Section~\ref{s8}.

\medskip

We would like to thank F.~Peherstorfer and P.~Yuditskii for the private communication \cite{Pehpc}.
J.S.C.\ would like to thank M.~Flach and A.~Lange for the hospitality of
Caltech where this work was completed.

%%%%%%%%%%%%%%%%%%%%%%%%%%%%%%%%%%%%%%%%%%%%%%%%%%%%%%%%%%%%%%
\section{Step-by-Step Sum Rules} \lb{s2}
%%%%%%%%%%%%%%%%%%%%%%%%%%%%%%%%%%%%%%%%%%%%%%%%%%%%%%%%%%%%%%

As noted in the introduction, a key to the approach to Szeg\H{o}-type theorems for $\fre=[-2,2]$ that we'll follow is
step-by-step sum rules. Our goal in this section is to prove those for a general finite gap $\fre$. In Theorem~7.5
of paper I, we proved such results for measures in the isospectral torus, and our discussion here will closely follow
the proof there. The major change is that there, with finitely many eigenvalues in $\bbR\setminus\fre$, we could use
finite Blaschke products. Here, because we do not wish to suppose a priori a $1/2$-power condition on the eigenvalues,
we'll need the alternating Blaschke products of Theorem~4.9 of paper I. Here is the result:

\begin{theorem}[Nonlocal step-by-step sum rule] \lb{T2.1} Let $J$ be a Jacobi matrix with $\sigma_\ess (J)
=\fre$. Let  $J^{(1)}$ be the once-stripped Jacobi matrix and let $\{p_j\}_{j=1}^\infty$ be the points in $\calF$
that are mapped by the covering map, $\x$, to the eigenvalues of $J$ and $\{z_j\}_{j=1}^\infty$ the corresponding
points for the eigenvalues of $J^{(1)}$. Let $B_\infty$ be the alternating Blaschke product with poles at
$\{\gamma(p_j)\}_{j=1; \gamma\in\Gamma}^\infty$ and zeros at $\{\gamma(z_j)\}_{j=1;\gamma\in\Gamma}^\infty$.
Let $B(z)$ be the Blaschke product with zeros at $\{\gamma(0)\}_{\gamma\in\Gamma}$. Let $M(z)$ be the $m$-function,
\eqref{1.21}, for $J$, and $M^{(1)}(z)$ the one for $J^{(1)}$. Then
\begin{SL}
\item[{\rm{(a)}}] $\lim_{r\uparrow 1} M(re^{i\theta})\equiv M(e^{i\theta})$ and $\lim_{r\uparrow 1} M^{(1)}
(re^{i\theta})\equiv M^{(1)} (e^{i\theta})$ exist for ${d\theta}/{2\pi}$-a.e. $\theta$.

\item[{\rm{(b)}}] Up to sets of ${d\theta}/{2\pi}$ measure zero,
\begin{equation} \lb{2.1}
\{\theta\mid \Ima M(e^{i\theta})\neq 0\}= \{\theta\mid\Ima M^{(1)} (e^{i\theta})\neq 0\}
\end{equation}

\item[{\rm{(c)}}]
\begin{equation} \lb{2.2}
\log\biggl( \f{\Ima M(e^{i\theta})}{\Ima M^{(1)} (e^{i\theta})}\biggr) \in \bigcap_{p<\infty} L^p
\biggl( \partial\bbD, \f{d\theta}{2\pi}\biggr)
\end{equation}

\item[{\rm{(d)}}] We have
\begin{equation} \lb{2.3}
a_1 M(z) = B(z) B_\infty(z) \exp\biggl( \int \f{e^{i\theta}+z}{e^{i\theta}-z}\,
\log \biggl( \f{\Ima M(e^{i\theta})}{\Ima M^{(1)}(e^{i\theta})}\biggr) \f{d\theta}{4\pi}\biggr)
\end{equation}
\end{SL}
\end{theorem}

\begin{remarks} 1. We've labeled the $p$'s and $z$'s to be infinite in number, although there may be only
finitely many. Moreover, we need to group them into not one sequence but potentially $2\ell+2$ if each of the
points in $\{\alpha_j, \beta_j\}_{j=1}^{\ell+1}$ is a limit point of eigenvalues in $\bbR\setminus\fre$. Once
this is done, one forms an alternating Blaschke product for each sequence (the $p$'s and $z$'s in each sequence
alternate along a boundary arc of $\calF$ or on $(0,1)$ or $(-1,0)$), and then takes the product of
these $2\ell +2$ alternating Blaschke products.

\smallskip
2. $\Ima M$ and $\Ima M^{(1)}$ have the same sign at each point of $\partial\bbD$, positive or negative, depending
on whether $\x$ maps to an upper or lower lip of $\fre$.

\smallskip
3. We've written (c) and (d) assuming that the set in \eqref{2.1} is all of $\partial\bbD$ (up to sets of Lebesgue
measure zero). A more proper version is that $\lim_{r\uparrow 1}\abs{M(re^{i\theta})}^2$ has a limit as $r\uparrow 1$
which, when multiplied by $a_1^2$, is the ratio  $\Ima M/\Ima M^{(1)}$ at points in the set in \eqref{2.1}. It is that
boundary value that enters in \eqref{2.2} and \eqref{2.3}.
\end{remarks}

\begin{proof} We follow the arguments used for Theorem~7.5 of paper I. For $z\in\bbD$, not a pole or zero of
$M$\!, let
\begin{equation} \lb{2.4}
h(z) = \f{a_1 M(z)}{B(z)B_\infty(z)}
\end{equation}
At the poles and zeros of $M$\!, $h(z)$ has removable singularities and no zero values, so $h$ is nonvanishing and
analytic in all of $\bbD$.

All of $M$\!, $B$, and $B_\infty$ are positive on $(0,\veps)$ for $\veps$ small, so one can choose a branch of
$\log(h(z))$ which has $\Ima (\log(h(z)))=0$ on $(0,\veps)$. Since $\Ima M>0$ on $\bbC_+\cap\calF$ and $\Ima M<0$
on $\bbC_-\cap\calF$, with this choice,
\begin{equation} \lb{2.5}
\abs{\arg(M(z))}\leq \pi \text{ on } \calF
\end{equation}
By eqn.\ (4.84) in Theorem~4.9 of paper I, there is a constant $C$ so that
\begin{equation} \lb{2.6}
\abs{\arg(B_\infty(z) B(z))} \leq C \text{ on } \calF
\end{equation}

As in the proof of Theorem~7.5 of paper I, this plus the fact that $h(z)$ is character automorphic implies that
\begin{equation} \lb{2.7}
\sup_{0<r<1} \int \abs{\Ima (\log(h(re^{i\theta})))}^p\, \f{d\theta}{2\pi} <\infty
\end{equation}
Thus, by the M.~Riesz theorem,
\begin{equation} \lb{2.8}
\log(h)\in\bigcap_{p<\infty} H^p (\bbD)
\end{equation}
This implies that $\log(h)$, and so $M$\!, has boundary values and
\begin{equation} \lb{2.9}
\log \abs{M(e^{i\theta})}\in\bigcap_{p<\infty} L^p \biggl(\partial\bbD, \f{d\theta}{2\pi}\biggr)
\end{equation}
Taking boundary values in (see \eqref{1.16})
\begin{equation} \lb{2.10}
M(z)^{-1} = \x(z) - b_1 - a_1^2 M^{(1)} (z)
\end{equation}
shows that \eqref{2.1} holds, and on the set where $\Ima M\neq 0$,
\begin{equation} \lb{2.11}
\abs{a_1 M(e^{i\theta})}^2 = \f{\Ima M(e^{i\theta})}{\Ima M^{(1)} (e^{i\theta})}
\end{equation}
This and \eqref{2.9} imply \eqref{2.2}, and \eqref{2.3} is just the Poisson representation for
$\log(h(z))$.
\end{proof}

The main use we'll make of \eqref{2.3} is to divide by $B(z)$ and take $z\to 0$ using \eqref{1.17} and
\eqref{1.20}. The result is:

\begin{theorem}[Step-by-step $C_0$ sum rule] \lb{T2.2}
\begin{equation} \lb{2.12}
\f{a_1}{\ca(\fre)} = B_\infty (0) \exp\biggl( \int_0^{2\pi}
\log\biggl( \f{\Ima M(e^{i\theta})}{\Ima M^{(1)} (e^{i\theta})}\biggr) \f{d\theta}{4\pi}\biggr)
\end{equation}
\end{theorem}

%%%%%%%%%%%%%%%%%%%%%%%%%%%%%%%%%%%%%%%%%%%%%%%%%%%%%%%%%%%%%%
\section{Fun and Games with Eigenvalues} \lb{s3}
%%%%%%%%%%%%%%%%%%%%%%%%%%%%%%%%%%%%%%%%%%%%%%%%%%%%%%%%%%%%%%

Sum rules include eigenvalue sums---it appears somewhat hidden in \eqref{2.12} as $B_\infty(0)$. Since, in exploiting
sum rules, we'll be looking at the behavior of sums over families, often with infinitely many elements, we'll
need control on such sums. This was true already in the single interval case as studied by \cite{KS,SZ},
but there the main tool needed was a simple variational principle. Eigenvalues above or below the essential
spectrum are given by a linear variational principle. This is not true for eigenvalues in gaps, and so we'll
need some extra techniques, which we put in the current section. We note that there are still limitations on
what can be done in gaps. For example, for perturbations of elements of the finite gap isospectral torus, there
is a $1/2$ critical Lieb--Thirring bound at the external edges \cite{FSW} but not yet one known for internal
gap edges \cite{HS}.

We begin with two results about the relation of eigenvalues of $J$ and $J^{(n)}$, the $n$-times stripped Jacobi
matrix of \eqref{1.14}.

\begin{theorem}\lb{T3.1} Let $J$ be a Jacobi matrix with $\sigma_\ess(J)=\fre$. Let $c\in (\beta_j, \alpha_{j+1})$,
one of the gaps of $\bbR\setminus\fre$. Suppose $f$ is defined, positive, and monotone on $(\beta_j,c)$ with
$\lim_{x\downarrow\beta_j} f(x)=0$. Let  $c>x_1(J) > x_2(J) >\cdots >\beta_j$ be the eigenvalues of $J$ in
$(\beta_j,c)$. Then the eigenvalues of $J$ and $J^{(1)}$ strictly interlace, that is, either
\begin{equation} \lb{3.1}
x_1(J) >x_1(J^{(1)}) > x_2(J) > x_2 (J^{(1)})> \dots
\end{equation}
or
\begin{equation} \lb{3.2}
x_1 (J^{(1)})>x_1(J) > x_2 (J^{(1)})>x_2(J) > \dots
\end{equation}
In particular, $\sum_{k=1}^\infty [f(x_k(J))-f(x_k(J^{(1)}))]$ is always conditionally convergent.
\end{theorem}

\begin{remarks} 1. For simplicity of notation, we stated this and the following theorem for $(\beta_j,c)$,
but a similar result holds for $(c,\alpha_{j+1})$ and also for $(-\infty,\alpha_1)$ and $(\beta_{\ell+1},\infty)$.

\smallskip
2. By iteration, we also get convergence of $\sum_{k=1}^\infty [f(x_k(J))-f(x_k(J^{(n)}))]$ for each $n$.
\end{remarks}

\begin{proof} By the fact that $x_k(J)$ are the poles of $m(z)$ in $(\beta_j,c)$ and $x_k(J^{(1)})$ the zeros,
and since $\f{d}{dz} m(z) = \int \f{d\mu(x)}{(x-z)^2} >0$ for $z\in (\beta_j,c)$, we see the interlacing,
which implies \eqref{3.1} (if $m(c)\leq 0$) or \eqref{3.2} (if $m(c)>0$). The conditional convergence
of the sum is standard for alternating sums.
\end{proof}

\begin{theorem}\lb{T3.2} Under the hypotheses of Theorem~\ref{T3.1}, if
\begin{equation} \lb{3.3}
S\equiv\sup_n \biggl|\, \sum_{k=1}^\infty f(x_k(J))-f(x_k(J^{(n)}))\biggr| <\infty
\end{equation}
then
\begin{equation} \lb{3.4}
\sum_{k=1}^\infty f(x_k(J)) <\infty
\end{equation}
\end{theorem}

\begin{proof} We will need the fact proven below (in Theorem~\ref{T3.4}) that for each $j\in\{1,\dots,\ell\}$ and $\veps >0$, there is an $N$ so for $n\geq N$, $J^{(n)}$ has
either $0$ or $1$ eigenvalue in $(\beta_j+\veps, \alpha_{j+1}-\veps)$.

So for $n\geq N$ we may have $x_1(J^{(n)})>\beta_j+\veps$, but
$x_k(J^{(n)})\leq \beta_j+\veps$ for all $k\geq 2$. Hence, for $n\geq N$,
\begin{align} \lb{3.5}
&\sum_{\{k\,\mid\, \beta_j+\veps < x_k(J) < c\}} [f(x_k(J))-f(\beta_j+\veps)] \no \\
&\quad\qquad \leq f(c)+\sum_{\{k\,\mid\, \beta_j+\veps < x_k(J) < c\}} [f(x_k(J))-f(x_k(J^{(n)}))]
%\leq f(c)+S
\end{align}
Recall now that $J^{(n)}$ can be obtained by decoupling $J$ with a rank 2 perturbation
(which is the sum of a positive and a negative rank 1 perturbation) and removing the finite block.
Therefore, if we pick $\veps>0$ so small that $x_3(J)>\beta_j+\veps$, it follows that $x_k(J)>x_k(J^{(n)})$
for all $k\geq 2$ (when $n\geq N$). This implies that
\begin{align}
\sum_{\{k\,\mid\, \beta_j+\veps < x_k(J) < c\}} [f(x_k(J))-f(x_k(J^{(n)}))]\leq S
\end{align}
So, for sufficiently small $\veps_0$ and $\veps_1 <\veps_0$,
\begin{equation} \lb{3.6}
\sum_{\{k\,\mid\, \beta_j+\veps_0 < x_k(J) < c\}} [f(x_k(J))-f(\beta_j+\veps_1)] \leq f(c)+S
\end{equation}
Taking $\veps_1\downarrow 0$ and then $\veps_0\downarrow 0$ yields \eqref{3.4}.
\end{proof}

The following lemma is well known, used for example in Denisov \cite{DenPAMS}:

\begin{lemma} \lb{L3.3} Let $A$ be a bounded operator with
\begin{equation} \lb{3.7}
\gamma =\inf (\sigma_\ess(A))
\end{equation}
Let $P_n$ be a family of orthogonal projections with
\begin{equation} \lb{3.8}
\slim P_n=0
\end{equation}
Then for any $\veps$, we can find $N$ so that for $n\geq N$,
\begin{equation} \lb{3.9}
\sigma(P_n AP_n\restriction \ran(P_n)) \subset [\gamma-\veps,\infty)
\end{equation}
\end{lemma}

\begin{proof} Since \eqref{3.7} holds, for any $\veps$, we can write
\begin{equation} \lb{3.10}
A=A_\veps + B_\veps
\end{equation}
where $A_\veps \geq \gamma -\veps/2$ and $B_\veps$ is finite rank, and so compact.

By \eqref{3.8}, $P_n B_\veps P_n\to 0$ in $\norm{\cdot}$, so we can find $N$ so that, for $n\geq N$,
$\norm{P_n B_\veps P_n}\leq \veps/2$. Then for each $n\geq N$,
\begin{equation} \lb{3.11}
P_n AP_n \geq P_n \biggl(\gamma - \f{\veps}{2} -\f{\veps}{2}\biggr) P_n \geq (\gamma-\veps)P_n
\end{equation}
proving \eqref{3.9}.
\end{proof}

\begin{theorem}\lb{T3.4} Let $J$ be a bounded Jacobi matrix with $(\alpha,\beta)\cap \sigma_\ess (J)=\emptyset$.
Let $J^{(n)}$ be the $n$-times stripped Jacobi matrix. Then for any $\veps$, we can find $N$\! so that, for
$n\geq N$, $J^{(n)}$ has at most one eigenvalue in $(\alpha+\veps,\beta-\veps)$.
\end{theorem}

\begin{proof} Let $P_n$ be the projection onto $\spann\{\delta_j\}_{j=n+1}^\infty$, so
\begin{equation} \lb{3.12}
J^{(n)} = P_n JP_n\restriction\ran(P_n)
\end{equation}
Let $\gamma=\f12(\alpha+\beta)$ and $A=(J-\gamma)^2$, $A^{(n)}=P_n AP_n\restriction\ran(P_n)$. By the
spectral mapping theorem,
\begin{equation} \lb{3.13}
\inf(\sigma_\ess (A))\geq [\tfrac12\, (\beta-\alpha)]^2
\end{equation}
so, by the lemma, for any $\veps'$, there is $N$ so for $n\geq N$,
\begin{equation} \lb{3.14}
\inf \sigma(A^{(n)}) \geq [\tfrac12\, (\beta-\alpha)]^2 -\veps' = [\tfrac12\, (\beta-\alpha)-\veps]^2
\end{equation}
where $\veps'$ is chosen so that \eqref{3.14} holds.

Since
\begin{equation} \lb{3.15}
A^{(n)} -(J^{(n)}-\gamma)^2 = P_n(J-\gamma)(1-P_n)(J-\gamma)P_n
\end{equation}
is rank one, $(J^{(n)}-\gamma)^2$ has at most one eigenvalue (which is simple) below $[\f12(\beta-\alpha)-\veps]^2$,
which proves the claimed result by the spectral mapping theorem.
\end{proof}

Next, we turn to estimating eigenvalue sums like
\begin{equation} \lb{3.16}
\calE(J) =\sum_{x\in\sigma(J)\setminus\fre} \dist(x,\fre)^{1/2}
\end{equation}
with a goal of showing, for example, that if $\calE(J)$ is finite, then so is $\sup_n \calE(J^{(n)})$.

\begin{definition} Let $A$ be a bounded selfadjoint operator with $(a,b)\cap\sigma_\ess(A)=\emptyset$.
We set
\begin{equation} \lb{3.17}
\Sigma_{(a,b)}(A) =\sum_{x\in\sigma(A)\cap(a,b)} \dist(x,\bbR\setminus(a,b))^{1/2}
\end{equation}
where the sum includes $x$ as many times as the multiplicity of that eigenvalue.
\end{definition}

\begin{theorem}\lb{T3.5} Let $A$ be a bounded selfadjoint operator with $(a,b)\cap\sigma_\ess(A)=\emptyset$
and $\Sigma_{(a,b)}(A)<\infty$. Then
\begin{SL}
\item[{\rm{(i)}}] If $B$ is another bounded selfadjoint operator with $\rank(B-A)=r <\infty$, then
\begin{equation} \lb{3.18}
\Sigma_{(a,b)}(B) \leq \Sigma_{(a,b)}(A) + r\biggl( \f{b-a}{2}\biggr)^{1/2}
\end{equation}

\item[{\rm{(ii)}}] If $P$ is an orthogonal projection so that $\rank(P\!A(1-P))=r<\infty$ and $B=P\!AP\restriction\ran(P)$,
then \eqref{3.18} holds.
\end{SL}
\end{theorem}

\begin{proof} For simplicity of notation, we can suppose $A$ has both $a$ and $b$ as limit points of eigenvalues
(from above and below, respectively). It is easy to modify the arguments if there are only finitely many eigenvalues.

\smallskip
(i) By induction, it suffices to prove this for $r=1$. Label the eigenvalues of $A$ in $(a,b)$, counting multiplicity,
by
\begin{equation} \lb{3.19}
a< \cdots \leq x_{-2}(A) \leq x_{-1}(A) < \tfrac12\, (a+b) \leq x_0(A)\leq x_1(A) \leq \cdots < b
\end{equation}
For $A$'s with a cyclic vector $\varphi$, and $B=A+\lambda(\varphi,\dott)\varphi$, it is well known that eigenvalues
of $A$ and $B$ strictly interlace. By writing $A$ as a direct sum of its restriction to the cyclic subspace for
$\varphi$ and the restriction to the orthogonal complement, we can label all the eigenvalues of $B$ in such
a way that
\begin{equation} \lb{3.20}
x_k(A)\leq x_{k+1}(B) \leq x_{k+1}(A)
\end{equation}

With that labeling,
\begin{align}
\sum_{k=1}^\infty \dist(x_k(B),\bbR\setminus(a,b))^{1/2}
&\leq \sum_{k=0}^\infty \dist(x_k(A),\bbR\setminus (a,b))^{1/2} \\
\sum_{k=1}^{\infty} \dist(x_{-k}(B),\bbR\setminus(a,b))^{1/2}
&\leq \sum_{k=1}^{\infty} \dist(x_{-k}(A),\bbR\setminus (a,b))^{1/2}
\end{align}
so that
\begin{equation} \lb{3.21}
\Sigma_{(a,b)}(B) \leq \dist(x_0(B),\bbR\setminus (a,b))^{1/2} + \Sigma_{(a,b)}(A)
\end{equation}
which implies \eqref{3.18} for $r=1$.

\smallskip
(ii) By scaling and adding a constant to $A$, we can suppose $b=-a=1$. For $C\geq0$ with $\sigma_\ess(C)\subset[1,\norm{C}]$,
let
\begin{equation} \lb{3.22}
\wti\Sigma(C)=\sum_{x\in\sigma(C)\cap [0,1)} \bigl( 1-\sqrt{x}\bigr)^{1/2}
\end{equation}
so that
\begin{equation} \lb{3.23}
\Sigma_{(-1,1)}(A) =\wti\Sigma(A^2)
\end{equation}

By mimicking the proof of (i), we see
\begin{equation} \lb{3.24}
\rank(D-C)=r,\, D\geq 0 \,\Rightarrow\, \wti\Sigma(D)\leq\wti\Sigma(C)+r
\end{equation}

Notice, next, that by the min-max principle, $x_k (PCP\restriction\ran(P))\geq x_k(C)$ so that
\begin{equation} \lb{3.25}
\wti\Sigma(PCP\restriction\ran(P))\leq\wti\Sigma(C)
\end{equation}
Notice also that
\begin{equation} \lb{3.26}
P\!A^2\!P-(P\!AP)^2 = P\!A(1-P)AP
\end{equation}
is at most rank $r$. Thus,
\begin{alignat*}{2}
\Sigma_{(-1,1)}(P\!AP\restriction\ran(P))
&=\wti\Sigma ((P\!AP\restriction\ran(P))^2) \qquad && \text{(by \eqref{3.23})} \\
&\leq r + \wti\Sigma (P\!A^2\!P\restriction\ran(P)) && \text{(by \eqref{3.24})} \\
&\leq r + \wti\Sigma(A^2) && \text{(by \eqref{3.25})} \\
&= r+\Sigma_{(-1,1)}(A) &&  \text{(by \eqref{3.23})}
\end{alignat*}
\end{proof}

We also want to know that one can make the eigenvalue sum small, uniformly in $B$, by summing only over
eigenvalues sufficiently near $a$ or $b$. Thus, we prove (for simplicity, we state the result for $a$;
a similar result holds for $b$):

\begin{theorem}\lb{T3.6} Let $(a,b)\cap\sigma_\ess(A)=\emptyset$, $\Sigma_{(a,b)}(A)<\infty$, and suppose $B$ is related to
$A$ as  in either {\rm{(i)}} or {\rm{(ii)}} of Theorem~\ref{T3.5}. Then for any $\delta <\f14(b-a)$,
\begin{equation} \lb{3.27}
\sum_{x_k(B)\in (a,a+\delta)} (x_k(B)-a)^{1/2} \leq r\delta^{1/2} +
\sum_{x_k(A)\in (a,a+2\delta)} (x_k(A)-a)^{1/2}
\end{equation}
\end{theorem}

\begin{proof} We have
\begin{align*}
\text{LHS of \eqref{3.27}}
&\leq \Sigma_{(a,a+2\delta)}(B) \\
&\leq \Sigma_{(a,a+2\delta)}(A) + r\delta^{1/2} \qquad \text{(by Theorem~\ref{T3.5})} \\
&=\text{RHS of \eqref{3.27}}
\qedhere
\end{align*}
\end{proof}

As a corollary, we have (since $J^{(n)}=P_n J P_n\restriction\ran(P_n)$ with $\rank((1-P_n)JP_n)=1$):

\begin{theorem}\lb{T3.7} Let $J$ be a Jacobi matrix with $\sigma_\ess(J)=\fre$. Given \eqref{3.16}, let $\calE(J)$
be finite and let $J^{(n)}$ be the $n$-times stripped Jacobi matrix. Then
\begin{SL}
\item[{\rm{(i)}}]
\begin{equation} \lb{3.28}
\calE(J^{(n)}) \leq\calE(J) + \ell \max_{j=1, \dots, \ell} \left(\tfrac12\, \abs{\alpha_{j+1}-\beta_j}\right)^{1/2}
\end{equation}

\item[{\rm{(ii)}}] For any $j\in\{1, \dots, \ell+1\}$ and $\veps>0$, there is a $\delta>0$ so that for all $n$,
\begin{align}
\sum_{x_k(J^{(n)})\in (\beta_j,\beta_j+\delta)} (x_k(J^{(n)}) - \beta_j)^{1/2} &\leq \tfrac12\, \veps \lb{3.29} \\
\sum_{x_k(J^{(n)})\in (\alpha_{j}-\delta,\alpha_{j})} (\alpha_{j} - x_k(J^{(n)}))^{1/2}
&\leq \tfrac12\, \veps \lb{3.30}
\end{align}
\end{SL}
\end{theorem}

\begin{proof} (i) By the min-max principle for eigenvalues above and below the essential spectrum, the
sums for eigenvalues below $\alpha_1$ or above $\beta_{\ell+1}$ get smaller. In each gap, we use
Theorem~\ref{T3.5} (ii). This yields \eqref{3.28} as $r=1$.

\smallskip
(ii) We prove \eqref{3.29}; the proof of \eqref{3.30} is similar. Take $\delta_0 <\f14 (\alpha_{j+1} -\beta_j)$
% or $\delta_0<1$ if $j=\ell+1$,
so that
\begin{equation}
\sum_{x_k(J)\in(\beta_j,\beta_j +2\delta_0)} (x_k(J)-\beta_j)^{1/2} < \tfrac14\, \veps
\end{equation}
Then pick $\delta <\delta_0$ so that $\delta^{1/2} <\f14\veps$. \eqref{3.27} implies \eqref{3.29}.
\end{proof}

\begin{theorem}\lb{T3.8} Let $J,\, \ti J$ be two Jacobi matrices with $\sigma_\ess(J)=\sigma_\ess(\ti J)=\fre$
and $\calE(J),\, \calE(\ti J)<\infty$. For $m,q\geq 0$, let $J_{m,q}$ be the Jacobi matrix with
\begin{align}
a_n(J_{m,q}) &= \begin{cases}
a_n(J) & n=1, \dots, m \\
a_{n-m+q}(\ti J) & n=m+1, \dots
\end{cases} \lb{3.31} \\
b_n(J_{m,q}) &= \begin{cases}
b_n(J) & n=1, \dots, m \\
b_{n-m+q}(\ti J) & n=m+1, \dots
\end{cases} \lb{3.32}
\end{align}
Then for a constant, $K$, independent of $m$ and $q$,
\begin{equation} \lb{3.33}
\calE(J_{m,q}) \leq \calE(J) + \calE(\ti J) + K
\end{equation}
and for any $j\in\{1,\dots,\ell+1\}$ and $\veps>0$, there is a $\delta>0$ so that for all $m,q$,
\begin{equation} \lb{3.34}
\sum_{x_k(J_{m,q})\in (\beta_j,\beta_j +\delta)} (x_k(J_{m,q})-\beta_j)^{1/2} < \tfrac12\, \veps
\end{equation}
A similar result holds near $\alpha_{j}$.
\end{theorem}

\begin{proof} Let $Q_m$ be the projection onto $\spann\{\delta_j\}_{j=1}^m$ and $P_m =1-Q_m$.
Then $J_{m,q}-Q_mJQ_m-P_m\ti J^{(q)} P_m$ is rank two. Thus, for $j=1, \dots, \ell$ and $\gamma =\max_{j=1, \dots, \ell}
(\f12\abs{\alpha_{j+1} -\beta_j})^{1/2}$,
\begin{align*}
\Sigma_{(\beta_j,\alpha_{j+1})}(J_{m,q})
&\leq 2\gamma + \Sigma_{(\beta_j,\alpha_{j+1})} (Q_m JQ_m) + \Sigma_{(\beta_j,\alpha_{j+1})} (P_m\ti J^{(q)}P_m) \\
&\leq 4\gamma + \Sigma_{(\beta_j,\alpha_{j+1})}(J) + \Sigma_{(\beta_j,\alpha_{j+1})} (\ti J^{(q)}) \\
&\leq 5\gamma + \Sigma_{(\beta_j,\alpha_{j+1})}(J) + \Sigma_{(\beta_j,\alpha_{j+1})} (\ti J)
\end{align*}

For eigenvalues below $\alpha_1$ (or above $\beta_{\ell+1}$), we use the fact that $\abs{a_n(J_{m,q})}\leq\norm{J}$
to see that $\norm{J_{m,q}}\leq 2\norm{J} + \norm{\ti J}$ (a crude over-estimate). Hence we can do a similar bound on some $\Sigma_{(\kappa,\alpha_1)}(J_{m,q})$ with $\kappa$ independent of $m$ and $q$.

The passage from the proof of \eqref{3.33} to the proof of \eqref{3.34} is similar to the argument in the
proof of Theorem~\ref{T3.7}.
\end{proof}

It is a well-known phenomenon that, under strong limits, spectrum can get lost (e.g., if $J_n$ is a Jacobi matrix
which is the free $J_0$, except that for $m\in (n^2-n, n^2 +n)$, $b_m=-2$, then $J_n \overset{s}{\longrightarrow}
J_0$ but $J_n$ has more and more eigenvalues in $(-4,-2)$). We are going to be interested in situations where
this doesn't happen, which is the last subject we consider in this section.

\begin{theorem}\lb{T3.9} Let $J$ be a Jacobi matrix with $\sigma_\ess(J)=\fre$. Suppose that $J^{(n_k)}\to
\ti J$ in the sense that for each $m\geq 1$,
\begin{equation} \lb{3.35}
a_{n_k+m}\to \ti a_m \qquad b_{n_k+m} \to \ti b_m
\end{equation}
Then $\ti J$ has at most one eigenvalue in $(\beta_j,\alpha_{j+1})$, and for each $\delta$ small and $n_k$ large,
$J^{(n_k)}$ has the same number of eigenvalues in $(\beta_j+\delta, \alpha_{j+1}-\delta)$ as $\ti J$. In fact, if
$\ti J$ has an eigenvalue $\ti \lambda$ there, the eigenvalue of $J^{(n_k)}$ in that interval converges to $\ti \lambda$.
\end{theorem}

\begin{proof} If $\ti \lambda$ is an eigenvalue of $\ti J$ in $(\beta_j, \alpha_{j+1})$ with $\ti J \ti u
=\ti \lambda \ti u$ (and $\norm{\ti u}=1$), then $\veps_{n_k}\equiv\norm{(J^{(n_k)}-\ti \lambda)\ti u}
\to 0$. Thus, $(\ti \lambda -\veps_{n_k}, \ti \lambda + \veps_{n_k})\cap\sigma(J^{(n_k)})\neq \emptyset$.
Since the interval for small enough $\veps_{n_k}$ is disjoint from $\sigma_\ess(J^{(n_k)})$, we conclude that there is at
least one eigenvalue $\lambda_{n_k}$ in the interval, and clearly, $\lambda_{n_k}\to\ti \lambda$.

This fact plus Theorem~\ref{T3.4} implies that $\ti J$ has at most one eigenvalue in $(\beta_j,\alpha_{j+1})$.

Suppose next that $J^{(n_k)} u_{n_k}=\lambda_{n_k} u_{n_k}$ with $\norm{u_{n_k}}=1$ and $\lambda_{n_k}\to\ti \lambda\in
(\beta_j,\alpha_{j+1})$. Given $v\in\ell^2(\bbN)$ and $n_k$, define
\begin{equation} \lb{3.36}
(v^{(n_k)})_m = \begin{cases}
0 & m\leq n_k \\
v_{m-n_k} & m>n_k
\end{cases}
\end{equation}
Then
\begin{equation} \lb{3.37}
\biggl[J v^{(n_k)}-(J^{(n_k)}v)^{(n_k)}\biggr]_m =
\begin{cases}
0 & m \neq n_k \\
a_{n_k} v_1 & m=n_k
%-a_{n_k} v_1 & m=n_{k+1}
\end{cases}
\end{equation}
We conclude that
\begin{equation} \lb{3.38}
\norm{(J-\lambda_{n_k})u_{n_k}^{(n_k)}} = a_{n_k}\abs{(u_{n_k})_1}
\end{equation}

If $(u_{n_k})_1\to 0$, this implies $\ti \lambda\in\sigma_\ess(J)$ since $u_{n_k}^{(n_k)}\overset{w}{\longrightarrow}0$.
But that is impossible, so $(u_{n_k})_1\nrightarrow 0$. By compactness of the unit ball in the weak topology, we
conclude $u_{n_k}$ has a weak limit point $\ti u$ with $(\ti u)_1\neq 0$, so $\ti u\not\equiv 0$. But $(\ti J
-\ti \lambda)\ti u=0$, so $\ti \lambda\in\sigma(\ti J)$.

We have thus proven the final sentence in the theorem, given Theorem~\ref{T3.4}, which says $J^{(n_k)}$ for $k$ large has
at most one eigenvalue in $(\beta_j+\delta, \alpha_{j+1}-\delta)$.
\end{proof}

The final theorem of the section deals with a specialized situation that we'll need later.

\begin{theorem}\lb{T3.10} Let $J$ be a Jacobi matrix with $\sigma_\ess (J)=\fre$. Suppose that, as $n_k\to\infty$, \eqref{3.35}
holds for some two-sided $\ti J$ and all $m\in\bbZ$.
Let $J_k$ be defined by
\begin{align}
a_n (J_k) &= \begin{cases}
a_m & m\leq n_k \\
\ti a_{m-n_k} & m> n_k
\end{cases} \\
b_m (J_k) &= \begin{cases}
b_m & m\leq n_k \\
\ti b_{m-n_k} & m> n_k
\end{cases}
\end{align}
Then for any $\delta >0$, with $\{\beta_j+\delta, \alpha_{j+1}-\delta\}\notin\sigma(J)$, all the eigenvalues of
$J_k$ in $(\beta_j+\delta,\alpha_{j+1}-\delta)$ for $k$ large are near eigenvalues of $J$ in that interval,
and these eigenvalues converge to those for $J$. Moreover, there is exactly one eigenvalue of $J_k$ near
a single eigenvalue of $J$ in that interval.
\end{theorem}

\begin{proof} We follow the first part of the proof of the last theorem until the analysis of $J_k u_k = \lambda_k u_k$ with $\lambda_k\to\lambda_\infty\in (\beta_j+\delta,\alpha_j-\delta)$. If we prove that $\lambda_\infty\in
\sigma(J)$ and $u_k$ converges in norm to the corresponding eigenvector, we are done. For we immediately get existence
of eigenvalues near $\lambda_\infty$, and uniqueness follows from the orthogonality of eigenvectors and the norm
convergence.

Define $\ti u_k\in\ell^2 (\bbZ)$ by
\begin{equation}
(\ti u_k)_m = \begin{cases}
(u_k)_{m+n_k} & m> -n_k \\
0 & m\leq -n_k
\end{cases}
\end{equation}
and suppose $\ti u_k$ has a nonzero weak limit $\ti u_\infty$. Then $(\ti J-\lambda_\infty)\ti u_\infty =0$,
so $\lambda_\infty\in\sigma(\ti J)$. As $\sigma(\ti J)\subset\sigma_\ess(J)=\fre$ by approximate eigenvector
arguments (see, e.g., \cite{LS}), we arrive at a contradiction. Thus, $\ti u_k$ converges weakly to zero.
This implies that its projection $P\ti u_k$ onto $\ell^2(\bbN)$ converges to zero in norm since otherwise
$\norm{(\ti J-\lambda_\infty)P \ti u_k} \to 0$ which is again impossible because $\lambda_\infty\notin\sigma(\ti J)$.

Therefore, we conclude that $\norm{(J-\lambda_\infty)u_k}\to 0$. Since $\lambda_\infty$ is a simple discrete point of
$\sigma(J)$, this can only happen if $\lambda_\infty$ is an eigenvalue of $J$ and $\norm{(1-P')u_k}\to 0$, where
$P'$ is the projection onto the eigenvector of $\lambda_\infty$; that is, $u_k$ converges to that eigenvector
in norm.
\end{proof}

%%%%%%%%%%%%%%%%%%%%%%%%%%%%%%%%%%%%%%%%%%%%%%%%%%%%%%%%%%%%%%
\section{Szeg\H{o}'s theorem} \lb{s4}
%%%%%%%%%%%%%%%%%%%%%%%%%%%%%%%%%%%%%%%%%%%%%%%%%%%%%%%%%%%%%%

Our goal in this section is the following. Let $\fre$ be a finite gap set, $J$ a bounded Jacobi matrix with
$\sigma_\ess(J)=\fre$, and $\{a_n,b_n\}_{n=1}^\infty$ its Jacobi parameters. Let $\{x_k\}$ be the
eigenvalues of $J$ outside $\fre$, and write
\begin{equation} \lb{4.1}
d\mu(x) =w(x)\, dx + d\mu_\s(x)
\end{equation}
where $d\mu$ is the spectral measure for $J$.

Next, define
\begin{equation} \lb{4.2}
A_n = \f{a_1 \cdots a_n}{\ca(\fre)^n} \qquad
\bar A = \limsup A_n \qquad
\ul{A\!} = \liminf A_n
\end{equation}
Consider the three conditions:
\begin{SL}
\item[(i)] Szeg\H{o} condition
\begin{equation} \lb{4.3}
\int_\fre \log(w(x)) \dist(x,\bbR\setminus\fre)^{-1/2}\, dx >-\infty
\end{equation}

\item[(ii)] Blaschke condition
\begin{equation} \lb{4.4}
\calE(J) = \sum_k \dist (x_k,\fre)^{1/2} <\infty
\end{equation}

\item[(iii)] Widom condition
\begin{equation} \lb{4.6}
0 < \ul{A\!} \, \leq \bar A <\infty
\end{equation}
\end{SL}

\begin{theorem} \lb{T4.1} Any two of {\rm{(i)--(iii)}} imply the third.
\end{theorem}

\begin{remarks} 1. We'll eventually prove more; for example, if (ii) holds, then (i) $\Leftrightarrow \bar A >0$;
and if either holds, then (iii) holds.

\smallskip
2. This is a precise analog of a result for $\fre =[-2,2]$ of Simon--Zlato\v{s} \cite{SZ} (cf.\ Theorem~\ref{T1.1})
who relied in part on Killip--Simon \cite{KS} and Simon \cite{S288}.

\smallskip
3. For $\fre =[-2,2]$, the relevance of \eqref{4.4} to Szeg\H{o}-type theorems is a discovery of Killip--Simon
\cite{KS} and Peherstorfer--Yuditskii \cite{PYpams}.

\smallskip
4. When there are no eigenvalues, the implication (i) $\Rightarrow$ (iii) is a result of Widom \cite{Widom}; see also
Aptekarev \cite{Apt}. Peherstorfer--Yuditskii \cite{PY} allowed infinitely many bound states, and in \cite{PYarx},
they proved (i) $\Rightarrow$ (iii) if (ii) holds. The other parts of Theorem~\ref{T4.1} are new, although as noted
to us by Peherstorfer and Yuditskii \cite{Pehpc}, there is an argument to go from \cite{PY,PYarx} to (iii) $\Rightarrow$ (i)
if (ii) holds (see Remark~3 following Theorem~\ref{T4.5} below).
\end{remarks}

Recall that, given any pair of Baire measures, $d\mu$, $d\nu$, on a compact Hausdorff space, we define their relative
entropy by
\begin{equation}
S(\mu\mid\nu) = \begin{cases}
-\infty &\text{ if $d\mu$ is not $d\nu$-a.c.} \\
-\int \log(\f{d\mu}{d\nu})\, d\mu &\text{ if $d\mu$ is $d\nu$-a.c.} \label{4.6a}
\end{cases}
\end{equation}
It is a fundamental fact (see, e.g., \cite[Thm.~2.3.4]{OPUC1}) that $S(\mu\mid\nu)$ is jointly concave and jointly
weakly upper semicontinuous in $d\mu$ and $d\nu$, and that
\begin{equation} \lb{4.7}
\mu(X) = \nu(X) =1 \,\Rightarrow\, S(\mu\mid\nu) \leq 0
\end{equation}

$S$ is relevant because we define
\begin{equation} \lb{4.8}
Z(J) = -\tfrac12\, S(\rho_\fre\mid\mu_J)
\end{equation}
with $d\mu_J$ the spectral measure of $J$ and $d\rho_\fre$ the potential theoretic equilibrium measure for $\fre$.
Then, by \eqref{4.7},
\begin{equation}
Z(J)\geq 0 \lb{4.9}
\end{equation}
More importantly,
\begin{equation}
\text{\eqref{4.3}} \,\Leftrightarrow\, Z(J) <\infty \lb{4.10}
\end{equation}
We have \eqref{4.10} because (see eqn.\ (4.31) and Theorem~4.4 of paper I) $d\rho_\fre$ is $dx\restriction\fre$ a.c.\
and
\begin{equation} \lb{4.11}
C_1\, \dist(x,\bbR\setminus\fre)^{-1/2} \leq \f{d\rho_\fre}{dx} \leq C_2\, \dist(x,\bbR\setminus\fre)^{-1/2}
\end{equation}
for $0<C_1<C_2<\infty$.

Given the connection \eqref{1.19x} between Blaschke products and $G_\fre$, the potential theoretic Green's
function for $\fre$, and the symmetry of Blaschke products (eqn.\ (4.19) of paper~I), one can rewrite the
step-by-step $C_0$ sum rule, Theorem~2.2, as

\begin{theorem}\lb{T4.2} For each $n$, $Z(J)<\infty \Leftrightarrow Z(J^{(n)})<\infty$, and in that case,
\begin{equation} \lb{4.4x}
\f{a_1\cdots a_n}{\ca(\fre)^n} = K_n \exp [Z(J^{(n)})-Z(J)]
\end{equation}
where
\begin{equation} \lb{4.5x}
K_n = \exp\biggl( \sum_k [G_\fre (x_k (J))-G_\fre (x_k (J^{(n)}))]\biggr)
\end{equation}
\end{theorem}

\begin{remark} By Theorem~\ref{T3.1}, and the monotonicity of $G_\fre$ near gap edges (eqns.\ (4.45) and
(4.46) of paper~I), the sum in \eqref{4.5x} is always conditionally convergent if ordered properly.
\end{remark}

\begin{proof} By iterating, it suffices to prove the result for $n=1$. As noted, $K_1$ is always finite and
the remarks before the statement of the theorem show that for $n=1$, $K_1=B_\infty (0)$. Thus, the step-by-step
$C_0$ sum rule says
\begin{equation} \lb{4.6x}
\f{a_1}{\ca(\fre)} = K_1 \exp\biggl( \f12 \int_0^{2\pi} \log
\biggl( \f{\Ima M(e^{i\theta})}{\Ima M^{(1)}(e^{i\theta})}\biggr)
\, \f{d\theta}{2\pi}\biggr)
\end{equation}

Since $M$ and so $\Ima M$ is automorphic, Corollary~4.6 of paper~I implies
\begin{equation} \lb{4.7x}
\int_0^{2\pi} \log \biggl( \f{\Ima M(e^{i\theta})}{\Ima M^{(1)}(e^{i\theta})}\biggr) \f{d\theta}{2\pi} =
\int_\fre \log \biggl( \f{w(x;J)}{w(x;J^{(1)})}\biggr) d\rho_\fre(x)
\end{equation}
where we use
\begin{equation} \lb{4.14a}
w(x;J)=\f{1}{\pi}\, \Ima m(x+i0,J)
\end{equation}
Thus,
\begin{equation}
\int_\fre \log (w(x;J^{(1)}))\, d\rho_\fre(x) > -\infty \,\Leftrightarrow\,
\int_\fre \log(w(x;J))\, d\rho_\fre(x) >-\infty
\end{equation}
showing $Z(J^{(1)}) <\infty \Leftrightarrow Z(J)<\infty$. Moreover, if both are finite,
\begin{equation} \lb{4.8x}
\text{RHS of \eqref{4.7x}} = 2Z(J^{(1)})-2Z(J)
\end{equation}
\eqref{4.6x}--\eqref{4.8x} imply \eqref{4.4x}.
\end{proof}

\begin{proposition} \lb{P4.3} We have that
\begin{equation} \lb{4.9x}
K_n \leq A_n e^{Z(J)}
\end{equation}
In particular, for some constant $C_1$,
\begin{equation} \lb{4.10x}
\ul{A\!} \,(J) \geq e^{-Z(J)} \liminf [\exp (-C_1 \calE (J^{(n)}))]
\end{equation}
and
\begin{equation} \lb{4.11x}
\limsup K_n \leq \bar A(J) e^{Z(J)}
\end{equation}
\end{proposition}

\begin{proof} \eqref{4.9x} is immediate from \eqref{4.4x} if we note that $Z(J^{(n)})\geq 0$ so that $\exp(-Z(J^{(n)}))\leq 1$. \eqref{4.10x} follows from noting that $K_n \geq \exp (-\sum_k G_\fre (x_k(J^{(n)})))$ since $G_\fre (x_k(J))
\geq 0$ and then, that for some $C_1$ (depending only on $\fre$),
\begin{equation} \lb{4.12x}
G_\fre (x) \leq C_1 \, \dist (x,\fre)^{1/2}
\end{equation}
by Theorem~4.4 of paper~I. Finally, \eqref{4.11x} is immediate by taking $\limsup$ in \eqref{4.9x}.
\end{proof}

\begin{proposition} \lb{P4.4} Let $J_\fre$ be the Jacobi matrix with spectral measure $d\rho_\fre$ and let $\{a_n^{(\fre)},
b_n^{(\fre)}\}_{n=1}^\infty$ be its Jacobi parameters. Let $J_n$ be the Jacobi matrix with parameters
\begin{align} \lb{4.13}
a_m (J_n) &= \begin{cases}
a_m & m=1, \dots, n\\
a_{m-n}^{(\fre)} & m>n
\end{cases} \\
b_m (J_n) &= \begin{cases}
b_m & m=1, \dots, n \\
b_{m-n}^{(\fre)} & m>n
\end{cases}
\end{align}
Then
\begin{equation} \lb{4.14}
A_n(J) = \exp\biggl( \sum_k G_\fre (x_k(J_n))\biggr) \exp (-Z(J_n))
\end{equation}
In particular, for some $C_1$ {\rm{(}}depending only on $\fre${\rm{)}},
\begin{equation} \lb{4.15}
A_n(J) \leq \exp (C_1\calE (J_n)-Z(J_n))
\end{equation}
\end{proposition}

\begin{proof} $J_n$ is defined so that
\begin{equation} \lb{4.16}
(J_n)^{(n)} = J_\fre
\end{equation}
and
\begin{equation} \lb{4.17}
A_n (J_n) = A_n (J)
\end{equation}

Thus, since $Z(J_\fre) =0$ and $J_\fre$ has no eigenvalues outside $\fre$, \eqref{4.4x} for $J_n$ is
\eqref{4.14}. \eqref{4.15} is then immediate from \eqref{4.12x}.
\end{proof}

\begin{theorem} \lb{T4.5} If $\calE(J)<\infty$, then
\begin{equation} \lb{4.18}
\bar A(J) >0 \,\Leftrightarrow\, Z(J) <\infty
\end{equation}
and if these are true, the Widom condition holds:
\begin{equation} \lb{4.19}
0 < \ul{A\!} \,(J) \leq \bar A(J) <\infty
\end{equation}
\end{theorem}

\begin{proof} By \eqref{4.10x} and Theorem~\ref{T3.7},
\begin{equation} \lb{4.20}
\calE(J), Z(J) <\infty \Rightarrow \ul{A\!} \,(J) >0 \Rightarrow \bar A(J) >0
\end{equation}

By \eqref{4.15} and Theorem~\ref{T3.8}, going through a subsequence with $A_{n_j}(J) \to \bar A(J)$, we see that
\begin{equation} \lb{4.21}
\calE(J)<\infty,\ \bar A(J)>0 \Rightarrow \limsup[\exp(-Z(J_{n_j}))]>0
\end{equation}
Thus, for some subsequence,
\begin{equation} \lb{4.22}
\liminf Z(J_{n_j})<\infty
\end{equation}
Since $J_{n_j}\overset{s}{\longrightarrow} J$, the spectral measures converge weakly. Since $S$ is upper semicontinuous,
$Z=-\f12 S$ is lower semicontinuous, and thus,
\begin{equation} \lb{4.23}
Z(J) \leq \liminf Z(J_{n_j})
\end{equation}
so \eqref{4.22} implies $Z(J)<\infty$. That is, we have proven
\begin{equation} \lb{4.24}
\calE(J) <\infty,\ \bar A(J) >0 \Rightarrow Z(J) <\infty
\end{equation}

If we have $Z(J) <\infty$ and $\calE(J) <\infty$, %and $\bar A(J) >0$,
we get $\ul{A}(J) >0$ by \eqref{4.20}, and since $Z(J_n) \geq 0$, \eqref{4.15} implies
\begin{equation} \lb{4.25}
\bar A(J) \leq \limsup [\exp (C_1 \calE(J_n))] <\infty
\end{equation}
by Theorem~\ref{T3.8}.
\end{proof}

\begin{remarks} 1. The above proof shows that even without $Z(J)<\infty$, we have $\calE(J)<\infty \Rightarrow
\bar A(J) <\infty$.

\smallskip
2. The proof borrows heavily from ideas of Killip--Simon \cite{KS} and Simon--Zlato\v{s} \cite{SZ}.

\smallskip
3. As noted, $\calE(J), Z(J)<\infty\Rightarrow$ \eqref{4.19} is a prior result (using variational
methods) of Peherstorfer--Yuditskii \cite{PY,PYarx}. Peherstorfer and Yuditskii \cite{Pehpc} have pointed out that their
results can be used to prove $\calE(J)<\infty,\ \bar A(J) >0 \Rightarrow Z(J)<\infty$ by the following
argument: While it is not explicitly stated, \cite{PY,PYarx} prove that for any $K$, there is a constant $C$ so that
for all measures with $Z(J)<\infty$ and $\calE(J)\leq K$,
\begin{equation} \lb{4.34}
\limsup_{n\to\infty} \f{a_1\cdots a_n}{\ca(\fre)^n} \leq Ce^{-Z(J)}
\end{equation}
Given $d\mu$ with $Z(J)=\infty$ and $\calE(J)\leq K$, let $d\ti\mu_\veps$ be the measure $d\mu+\veps\,
dx\restriction\fre$. Then with $d\mu_\veps$ the normalized measure and $a_n(\veps)$ the corresponding
$a$'s, \eqref{4.34} implies (since $Z(J_\veps)<\infty$)
\begin{equation} \lb{4.35}
\limsup \f{a_1(\veps)\cdots a_n(\veps)}{\ca(\fre)^n} \leq Ce^{-Z(J_\veps)}
\end{equation}
By the variational principle for $a_1\cdots a_n =\norm{P_n}$, we have
\begin{equation} \lb{4.36}
a_1\cdots a_n \leq [a_1 (\veps) \cdots a_n (\veps)] (1+\veps\abs{\fre})^{1/2}
\end{equation}
Since $Z(J_\veps) -\f12\log (1+\veps\abs{\fre})\uparrow Z(J)$, \eqref{4.35}--\eqref{4.36} imply that
$\bar A(J)=0$ if $Z(J)=\infty$. This argument for the classical Szeg\H{o} case is in Garnett \cite{Garn}.
\end{remarks}

\begin{theorem}\lb{T4.6} $\bar A(J), Z(J) <\infty \Rightarrow \calE(J) <\infty$
\end{theorem}

\begin{proof} This is immediate from \eqref{4.11x} and Theorem~\ref{T3.2}.
\end{proof}

\begin{remark} This argument follows ideas of Simon--Zlato\v{s} \cite{SZ}.
\end{remark}

Theorems~\ref{T4.5} and \ref{T4.6} imply Theorem~\ref{T4.1}.

%%%%%%%%%%%%%%%%%%%%%%%%%%%%%%%%%%%%%%%%%%%%%%%%%%%%%%%%%%%%%%
\section{Jost Functions and Jost Solutions} \lb{s5}
%%%%%%%%%%%%%%%%%%%%%%%%%%%%%%%%%%%%%%%%%%%%%%%%%%%%%%%%%%%%%%

In Section~8 of paper~I, we defined the Szeg\H{o} class for $\fre$, which we'll denote $\Sz(\fre)$, to be
the set of probability measures, $d\mu$, of the form \eqref{4.1} that obey \eqref{4.3} and \eqref{4.4}.
As usual, we associate $d\mu$ with its Jacobi matrix and Jacobi parameters $\{a_n,b_n\}_{n=1}^\infty$,
which we will write as $\{a_n(\mu), b_n(\mu)\}_{n=1}^\infty$ if we need to be explicit about the
measures. Of course, the $a$'s obey the Widom condition \eqref{4.6} for all measures in the Szeg\H{o} class.

In this section, we want to recall the definitions of Jost function and Jost solution from Sections~8
and 9 of paper~I, extend some results on Jost solutions to the full Szeg\H{o} class, and state the main
theorem that we'll prove in the next section about their asymptotics.

Jost functions require a reference measure, and we'll use the one from paper~I. Let $\ti\zeta_j\in\wti C_j^+$,
the full orthocircle, be the point farthest from $0$ on $\wti C_j^+$ and let $w_j\in\calS$, the Riemann surface
for $\fre$, be given by $w_j=\x^\sharp (\ti\zeta_j)$. Each $w_j$ lies in $G_j =\pi^{-1}([\beta_j, \alpha_{j+1}])$,
so $\vw =(w_1, \dots, w_\ell)\in\bbG =G_1 \times \cdots \times G_\ell$, which can be associated with the
isospectral torus. Our reference measure is the measure in $\calT_\fre$ associated to $\vw$. We denote it by
\begin{equation} \lb{5.0}
d\nu_\fre(x)=v_\fre(x)\,dx
\end{equation}
We point out that while our choice of the reference measure is convenient, one can take any other measure
in the Szeg\H{o} class to be the reference measure.

Given $d\mu\in\Sz(\fre)$, let $\{x_k\}$ be the eigenvalues of $J$ in $\bbR\setminus\fre$ and define $z_k\in\calF$ by
\begin{equation} \lb{5.1}
\x(z_k) = x_k
\end{equation}
The Jost function is then defined on $\bbD$ by
\begin{equation} \lb{5.2}
u(z;\mu) = \prod_{k} B(z,z_k) \exp \biggl(\f{1}{4\pi} \int_0^{2\pi} \f{e^{i\theta}+z}{e^{i\theta}-z}\, \log
\biggl(\f{v_\fre(\x(e^{i\theta}))}{w(\x(e^{i\theta}))}\biggr) d\theta\biggr)
\end{equation}
Since \eqref{4.14a} implies
\begin{equation} \lb{5.3}
\f{v_\fre(\x(e^{i\theta}))}{w(\x(e^{i\theta}))} = \f{\Ima M_{\nu_\fre}(e^{i\theta})}{\Ima M_\mu(e^{i\theta})}
\end{equation}
we could use that ratio instead. By the Blaschke condition and Proposition~4.8 of paper~I, the product in
\eqref{5.2} (which we'll call the Blaschke part) converges. By eqn.~(4.54) of paper~I and
the Szeg\H{o} condition for $d\mu$ and $d\nu_\fre$, the $\log$ in \eqref{5.2} is in $L^1 (\partial\bbD, d\theta/2\pi)$.
We call the exponential in \eqref{5.2} the Szeg\H{o} part. As proven in Theorem~8.2 of paper~I, $u$ is
a character automorphic function on $\bbD$.

For any Jacobi matrix, $J$, with $\sigma_\ess(J)=\fre$, we let $M^{(n)}$ be the $m$-function \eqref{1.21} of the $n$-times
stripped Jacobi matrix, $J^{(n)}$, and define the {\it Weyl solution\/} by
\begin{equation} \lb{5.4}
W_n(z) = M(z) (a_1 M^{(1)}(z)) \cdots (a_{n-1} M^{(n-1)}(z))
\end{equation}
$M^{(k)}$ has poles at the inverse images of eigenvalues of $J^{(k)}$ and zeros at the inverse images of
eigenvalues of $J^{(k+1)}$, so there is a cancellation, and $W_n$ can be defined as meromorphic on $\bbD$
with poles exactly at the points $\zeta$ with $\x(\zeta)$ an eigenvalue of $J$.

The name, Weyl solution, comes from the fact that because $m$ is a ratio of solutions $L^2$ at $n=+\infty$,
$W_n$ obeys
\begin{equation} \lb{5.4x}
W_n(z) = -\jap{\delta_n, (J-\x(z))^{-1}\delta_1}
\end{equation}
so that for $k\geq 2$,
\begin{equation} \lb{5.5}
[(J-\x(z)) W_\bddot (z)]_k =0
\end{equation}
where $W_\bddot(z)$ is the vector $(W_1(z),W_2(z), \dots)$. That is,
\begin{equation} \lb{5.6}
a_n W_n(z) + b_{n+1} W_{n+1}(z) + a_{n+1} W_{n+2}(z) = \x(z) W_{n+1}(z)
\end{equation}
for $n=1,2,\dots$.

The Jost solution is defined by
\begin{equation} \lb{5.7}
u_n(z;\mu) = u(z;\mu) W_n(z)
\end{equation}
Since $u(z;\mu)$ is $n$-independent, \eqref{5.6} holds for $u_n$ also. Since $u$ has zeros at the points where
$M$, and so $W_n$, has poles, $u_n$ is analytic on $\bbD$.

\begin{theorem}\lb{T5.1}
\begin{equation} \lb{5.8}
a_n M^{(n-1)}(z)=B(z)\, \f{u(z;\mu_n)}{u(z;\mu_{n-1})}
\end{equation}
where $M^{(0)}=M$, $d\mu_{0}=d\mu$, and $d\mu_n$, $M^{(n)}$ are associated to $J^{(n)}$, the $n$-times
stripped Jacobi matrix.
\end{theorem}

\begin{proof} This is a rewrite of \eqref{2.3} for $J^{(n-1)}$.
\end{proof}

\begin{theorem}\lb{T5.2} Let $d\mu\in\Sz(\fre)$. Then
\begin{equation} \lb{5.9}
u_n(z;\mu) = a_n^{-1} B(z)^n u(z;\mu_n)
\end{equation}
where $d\mu_n$ is the spectral measure for $J^{(n)}$, the $n$-times stripped Jacobi matrix.
\end{theorem}

\begin{proof} By \eqref{5.8} and \eqref{5.4},
\begin{equation} \lb{5.10}
a_n W_n(z) = B(z)^n\, \f{u(z;\mu_n)}{u(z;\mu)}
\end{equation}
which by \eqref{5.7} implies \eqref{5.9}.
\end{proof}

The key asymptotic result of the next section is the following:

\begin{theorem}\lb{T5.3} Suppose $d\mu\in\Sz(\fre)$ and that for some subsequence $n_j\to\infty$ and all $m\in\bbZ$,
\begin{equation} \lb{5.11}
a_{n_j+m}(J_\mu) \to a_m^\sharp \qquad
b_{n_j+m}(J_\mu) \to b_m^\sharp
\end{equation}
for some point $\{a_n^\sharp, b_n^\sharp\}_{n=-\infty}^\infty$ in the isospectral torus. If $d\mu^\sharp$ is the
spectral measure for the Jacobi matrix with parameters $\{a_n^\sharp, b_n^\sharp\}_{n=1}^\infty$, then
\begin{equation} \lb{5.12}
u(z;\mu_{n_j}) \to u(z;\mu^\sharp)
\end{equation}
uniformly on compact subsets of $\bbD$.
\end{theorem}

We note, as will be explained in the next section, that there is no loss in supposing that
the limit $J^\sharp$ is
in the isospectral torus. We'll also show that Theorem~\ref{T5.3} allows the proof of \eqref{1.13} for
a point $\ti J$ in the isospectral torus.

%%%%%%%%%%%%%%%%%%%%%%%%%%%%%%%%%%%%%%%%%%%%%%%%%%%%%%%%%%%%%%
\section{Jost Asymptotics} \lb{s6}
%%%%%%%%%%%%%%%%%%%%%%%%%%%%%%%%%%%%%%%%%%%%%%%%%%%%%%%%%%%%%%

In this section, we'll prove Theorem~\ref{T5.3}, use this result to prove that for $d\mu\in\Sz(\fre)$,
the Jacobi parameters $a_n$, $b_n$ are asymptotic to a fixed element of $\calT_\fre$, and prove an
asymptotic formula for the Jost solution.

The key to our proof of the existence of an $\{\ti a_n,\ti b_n\}_{n=1}^\infty$ obeying \eqref{1.13} is the
Denisov--Rakhmanov--Remling theorem for $\fre$ (\cite{Remppt}) which implies that any right limit of $J$
lies in the isospectral torus. Tracking the characters of the Jost functions will determine exactly
which right limits occur. This leads to a proof quite different from the variational approach of \cite{Widom,Apt,PY}.

We write
\begin{equation} \lb{6.1}
u(z;\mu) = \beta(z;\mu) \veps (z;\mu)
\end{equation}
where $\beta$ is the Blaschke part and $\veps$ the Szeg\H{o} part. We'll prove \eqref{5.12} by proving separately
the convergence of the two parts.

\begin{theorem}\lb{T6.1} Under the hypotheses of Theorem~\ref{T5.3}, uniformly on compact subsets of $\bbD$,
\begin{equation} \lb{6.2}
\beta(z;\mu_{n_j}) \to \beta(z;\mu^\sharp)
\end{equation}
\end{theorem}

\begin{proof} By Theorem~\ref{T3.7} of this paper and Proposition~4.8 of paper~I (and its proof), given a compact
set $K\subset\bbD$ and $\veps>0$, we can find $\delta>0$ so that the product of the contributions to $\beta$ from $x$'s
with $\dist(x,\fre) <\delta$ are within $\veps$ of $1$ for all $z\in K$. Thus, it suffices to prove convergence of
individual $x$'s for $\mu_{n_j}$ to those for $\mu^\sharp$, and this follows from Theorem~\ref{T3.9}.
\end{proof}

To control the Szeg\H{o} part, we first need the following lemma of Simon--Zlato\v{s} \cite{SZ}:

\begin{theorem}[\cite{SZ}] \lb{T6.2} Let $X$ be a compact Hausdorff measure space, $d\rho$,  $d\mu_n$, $d\mu_\infty$
probability measures with $d\mu_n \to d\mu_\infty$ weakly, and
\begin{equation} \lb{6.3}
d\mu_n = f_n\, d\rho + d\mu_{n;\s}
\end{equation}
Suppose that
\begin{equation} \lb{6.4}
S(\rho\mid\mu_n) \to S(\rho\mid\mu_\infty)
\end{equation}
with all relative entropies finite. Then
\begin{equation} \lb{6.5}
\log(f_n)\, d\rho \,\overset{w}{\longrightarrow}\, \log(f_\infty)\, d\rho
\end{equation}
\end{theorem}

\begin{proof} If $h$ is continuous and strictly positive, by upper semicontinuity,
\begin{equation} \lb{6.6}
\limsup S(h\rho\mid\mu_n) \leq S(h\rho\mid\mu_\infty)
\end{equation}
or
\begin{equation} \lb{6.7}
\limsup\int \log(f_n h^{-1}) h\, d\rho \leq \int \log(f_\infty h^{-1})h\, d\rho
\end{equation}
so that
\begin{equation} \lb{6.8}
\limsup \int \log(f_n) h\, d\rho \leq \int \log(f_\infty) h\, d\rho
\end{equation}
For arbitrary continuous real-valued $g$, let $h=2\norm{g}_\infty \pm g$ to get
\begin{equation} \lb{6.9}
\lim\int\log(f_n)g\, d\rho = \int\log (f_\infty)g\, d\rho
\end{equation}
\end{proof}

\begin{proposition}\lb{P6.3} To get
\begin{equation} \lb{6.10}
\veps(z;\mu_{n_j})\to\veps (z;\mu^\sharp)
\end{equation}
uniformly for $z$ in compact subsets of $\bbD$, it suffices to prove that
\begin{equation} \lb{6.11}
\lim_{j\to\infty} S(\rho_\fre\mid\mu_{n_j}) = S(\rho_\fre\mid\mu^\sharp)
\end{equation}
\end{proposition}

\begin{proof} By definition of $\veps$, it suffices that as measures on $\partial\bbD$,
\[
\log\Big(\f{1}{\pi}\abs{\Ima M_{\mu_{n_j}}(e^{i\theta})}\Big)\f{d\theta}{2\pi}
\,\overset{w}{\longrightarrow}\,
\log\Big(\f{1}{\pi}\abs{\Ima M_{\mu^\sharp}(e^{i\theta})}\Big)\f{d\theta}{2\pi}
\]
Given $g\in C(\partial\bbD)$, define
\begin{equation} \lb{6.12}
\ti g(e^{i\theta})= \f{\sum_{\gamma\in\Gamma}g(\gamma(e^{i\theta}))\abs{\gamma'(e^{i\theta})}}
{\sum_{\gamma\in\Gamma} \abs{\gamma'(e^{i\theta})}}
\end{equation}
and $h$ on $\fre$ by
\begin{equation} \lb{6.13}
h(\x(e^{i\theta})) = \tfrac12\, [\ti g(e^{i\theta})+\ti g(e^{-i\theta})]
\end{equation}
Note that $h$ is continuous on $\fre$ since $\ti g$ is continuous on $\partial\calF\cap\partial\bbD$ by
eqn.\ (3.4) of paper I.

By Corollary~4.6 of paper~I,
\begin{equation} \lb{6.14}
\int_0^{2\pi} g(e^{i\theta}) \log\biggl( \f{1}{\pi}\, \abs{\Ima M_\mu (e^{i\theta})}\biggr)\f{d\theta}{2\pi}
= \int_\fre h(x) \log(w_\mu(x))\, d\rho_\fre(x)
\end{equation}
so the necessary weak convergence on $\partial\bbD$ is implied by weak convergence of $\log(f_{n_j})\,
d\rho_\fre$ to $\log(f_\infty)\, d\rho_\fre$. This in turn follows from \eqref{6.11} and
Theorem~\ref{T6.2}.
\end{proof}

\begin{theorem}\lb{T6.4} Under the hypotheses of Theorem~\ref{T5.3}, uniformly on compact subsets of $\bbD$,
\begin{equation} \lb{6.15}
\veps(z;\mu_{n_j})\to \veps(z;\mu^\sharp)
\end{equation}
\end{theorem}

\begin{proof} By Proposition~\ref{P6.3}, it suffices to prove \eqref{6.11}. Since $\mu_{n_j}
\overset{w}{\longrightarrow}\mu^\sharp$, upper semicontinuity of $S$ implies
\begin{equation} \lb{6.16}
\limsup S(\rho_\fre\mid\mu_{n_j}) \leq S(\rho_\fre\mid\mu^\sharp)
\end{equation}
So it suffices to prove that
\begin{equation} \lb{6.17}
\ul{S}\equiv\liminf  S(\rho_\fre\mid\mu_{n_j}) \geq S(\rho_\fre\mid\mu^\sharp)
\end{equation}

Pick a subsequence (that we'll still denote by $n_j$) so that $S(\rho_\fre\mid\mu_{n_j})\to\ul{S}$
%\liminf_{\text{original subsequence}} S(\rho_\fre\mid\mu_{n_j})$
and so that $\tau_j\to\tau_\infty$
for some $\tau_\infty>0$, where
\begin{equation} \lb{6.18}
\tau_j = \f{a_1 \cdots a_{n_j}}{\ca(\fre)^{n_j}}
\end{equation}
Note that by Theorem~\ref{T4.1} and $d\mu\in\Sz(\fre)$, the original $\tau_j$'s are bounded, so we can pick such
a convergent subsequence.

For $k<\ell$, let $J_{k,\ell}$ be the Jacobi matrix obtained by starting with $J^{(n_k)}$ and then putting $J^\sharp$
at sites beyond $n_\ell$, that is,
\begin{align}
a_m (J_{k,\ell}) &= \begin{cases}
a_{n_k+m} & 1\leq m \leq n_\ell -n_k \\
a_{m-n_\ell+n_k}^\sharp & m>n_\ell - n_k
\end{cases}  \lb{ 6.19} \\
b_m (J_{k,\ell}) &= \begin{cases}
b_{n_k+m} & 1 \leq m \leq n_\ell -n_k \\
b_{m-n_\ell+n_k}^\sharp & m > n_\ell - n_k
\end{cases} \lb{6.20}
\end{align}

Thus, $(J_{k,\ell})^{(n_\ell -n_k)} = J^\sharp$, so the iterated step-by-step $C_0$ sum rule says that
\begin{equation} \lb{6.21}
\f{\tau_\ell}{\tau_k} = \f{\beta(0;\mu^\sharp)}{\beta(0;\mu_{k,\ell})}\,
\exp \bigl[\tfrac12\, S(\rho_\fre\mid\mu_{k,\ell})-\tfrac12\, S(\rho_\fre\mid \mu^\sharp)\bigr]
\end{equation}
We claim that
\begin{equation} \lb{6.22}
\lim_{\ell\to\infty} \beta(0;\mu_{k,\ell}) = \beta(0;\mu_{n_k})
\end{equation}
Accepting this for now, we take $\ell\to\infty$ in \eqref{6.21}, using the upper semicontinuity of $S(\rho_\fre\mid\mu)$
in $\mu$ to get
\begin{equation} \lb{6.23}
\exp\bigl[\tfrac12\, S(\rho_\fre\mid\mu_{n_k})-\tfrac12\, S(\rho_\fre\mid\mu^\sharp)\bigr] \geq \f{\tau_\infty}{\tau_k}\,
\f{\beta(0;\mu_{n_k})}{\beta(0;\mu^\sharp)}
\end{equation}

Now take $k\to\infty$ using the assumption that $S(\rho_\fre\mid\mu_{n_k})\to\ul{S}$. Since $\tau_\infty/\tau_k\to 1$ and,
by \eqref{6.2},
\[
\frac{\beta(0;\mu_{n_k})}{\beta(0;\mu^\sharp)}\to 1
\]
we get \eqref{6.17}.

Thus, we need only prove \eqref{6.22}, which follows the proof of Theorem~\ref{T6.1}, but using
Theorems~\ref{T3.8} and \ref{T3.10}.
\end{proof}

\begin{proof}[Proof of Theorem~\ref{T5.3}] By \eqref{6.1}, this follows from Theorems~\ref{T6.1} and \ref{T6.4}.
\end{proof}

We can now prove \eqref{1.13}.

\begin{theorem}\lb{T6.5} Let $d\mu\in\Sz(\fre)$. Take $d\ti\mu$ to be the unique element in $\calT_\fre$ so that $u(z;\mu)$
and $u(z;\ti\mu)$ have the same automorphic character. Then
\begin{equation} \lb{6.24}
\lim_{n\to\infty}\, \abs{a_n-\ti a_n} + \abs{b_n-\ti b_n} =0
\end{equation}
\end{theorem}

\begin{remark} The existence and uniqueness of $d\ti\mu\in\calT_\fre$ follows from
Theorem~7.3 of paper~I.
\end{remark}

\begin{proof} If not, by compactness, there is a right limit $J^\sharp$ so that
\begin{equation} \lb{6.25}
a_{m+n_j}\to a_m^\sharp \qquad b_{m+n_k} \to b_m^\sharp
\end{equation}
and so that
\begin{equation} \lb{6.26}
\ti a_{m+n_j} \to a_m^{(\infty)} \qquad \ti b_{m+n_j}\to b_m^{(\infty)}
\end{equation}
with
\begin{equation} \lb{6.27}
J^\sharp\neq J^{(\infty)}
\end{equation}

By the Denisov--Rakhmanov--Remling theorem \cite{Remppt}, $J^\sharp$ and $J^{(\infty)}$ lie in the isospectral torus.
Let $\chi_B(\gamma)$ be the automorphic character of $B(z)$. Then with $\chi_J(\gamma)$ the character of the Jost
function for $J$, \eqref{5.8} and the fact that $M^{(n-1)}$ is automorphic implies that
\begin{equation} \lb{6.28}
\chi_{J^{(n)}} = \chi_J \chi_B^{-n} \qquad \chi_{\ti J^{(n)}} = \chi_{\ti J} \chi_B^{-n}
\end{equation}
Since the definition of $\ti J$ is $\chi_{\ti J}=\chi_J$, we see that
\begin{equation} \lb{6.28x}
\chi_{J^{(n)}} =\chi_{\ti J^{(n)}}
\end{equation}

By Theorem~\ref{T5.3} and the fact that uniform convergence of character automorphic functions implies convergence of
their characters, we get
\begin{equation} \lb{6.29}
\chi_{J^\sharp} = \chi_{J^{(\infty)}}
\end{equation}
But $J^\sharp$ and $J^{(\infty)}$ lie in the isospectral torus, so by Theorem~7.3 of paper~I,
\begin{equation} \lb{6.30}
J^\sharp = J^{(\infty)}
\end{equation}
This contradiction to \eqref{6.27} implies that \eqref{6.24} holds.
\end{proof}

As a corollary, we get convergence of Jost solutions.

\begin{theorem}\lb{T6.6} Uniformly on compact subsets of $\bbD$,
\begin{equation} \lb{6.31}
\frac{u_n(z;\mu)-u_n (z;\ti\mu)}{B(z)^{n}} \to 0
\end{equation}
Moreover,
\begin{equation} \lb{6.32}
\f{u_n(z;\mu)}{u_n (z;\ti\mu)} \to 1
\end{equation}
uniformly on compact subsets of $\calF^\intt$.
\end{theorem}

\begin{remark} At each point in $\{\gamma(0)\mid\gamma\in\Gamma\}$, $u_n$ and $B^n$ have zeros of order $n$,
so $u_n B^{-n}$ has removable singularities at those points.
\end{remark}

\begin{proof} Since $J^{(n)}$ and $\ti J^{(n)}$ (by Theorem~\ref{T6.5}) have the same right limits, by Theorem~\ref{T5.3},
\begin{equation} \lb{6.33}
\abs{u(z;\mu_n) - u(z;\ti \mu_n)}\to 0
\end{equation}
uniformly on $\bbD$. Since $a_n/\ti a_n\to 1$, \eqref{5.9} implies \eqref{6.31}.

As $u_n(z;\ti\mu)$ is bounded away from zero (uniformly in $n$) on compact subsets of $\calF^\intt$, \eqref{6.33}
implies \eqref{6.32}.
\end{proof}

\begin{corollary} \lb{C6.7} Let $d\mu\in\Sz(\fre)$ and let $d\ti\mu\in\calT_\fre$ be the measure for which
\eqref{6.24} holds. Then, as $n\to\infty$,
\begin{equation} \lb{6.34}
\f{a_1\cdots a_n}{\ti a_1 \cdots \ti a_n} \to \f{u(0;\ti\mu)}{u(0;\mu)}
\end{equation}
In particular, $a_1\cdots a_n/\ca(\fre)^n$ is asymptotically almost periodic.
\end{corollary}

\begin{proof} The final sentence follows from \eqref{6.34} and Corollary~7.4 of paper~I. To obtain \eqref{6.34}, note that
\eqref{5.8} at $z=0$ and \eqref{2.12} implies
\begin{equation} \lb{6.35}
\f{u(0;\mu_n)}{u(0;\mu)} = \f{a_1\cdots a_n}{\ca(\fre)^n}
\end{equation}
Thus,
\begin{equation} \lb{6.36}
\f{a_1\cdots a_n}{\ti a_1 \cdots \ti a_n} = \f{u(0;\ti\mu)}{u(0;\mu)}\, \f{u(0;\mu_n)}{u(0;\ti\mu_n)}
\end{equation}

Since $u(0;\nu)$ is bounded away from $0$ as $d\nu$ runs through the isospectral torus, \eqref{6.33} implies that
\[
\f{u(0;\mu_n)}{u(0;\ti\mu_n)} \to 1
\]
proving \eqref{6.34}.
\end{proof}

%%%%%%%%%%%%%%%%%%%%%%%%%%%%%%%%%%%%%%%%%%%%%%%%%%%%%%%%%%%%%%
\section{Szeg\H{o} Asymptotics} \lb{s7}
%%%%%%%%%%%%%%%%%%%%%%%%%%%%%%%%%%%%%%%%%%%%%%%%%%%%%%%%%%%%%%

In Section~\ref{s6}, we proved that if $u_n$ is the Jost solution of a $J_\mu$ with $d\mu\in\Sz(\fre)$ and
$\ti u_n$ is the Jost solution for the element of the isospectral torus to which $J_\mu$ is asymptotic (in
the sense of \eqref{1.13}), then, as $n\to\infty$,  $u_n(z)/\ti u_n(z)\to 1$ uniformly on compact subsets of
$\calF^\intt$. Our goal in this section is to prove that if $p_n$ and $\ti p_n$ are the
corresponding orthonormal polynomials, then also on $\calF^\intt$, $p_n(z)/\ti p_n(z)$ has a limit
(which will not be identically $1$ and which we'll write explicitly in terms of Jost functions).

The passage from Jost asymptotics to Szeg\H{o} asymptotics in the case $\fre=[-2,2]$ was studied by Damanik--Simon
\cite{Jost1} using constancy of the Wronskian. Our first approach for general $\fre$ mimicked that of \cite{Jost1}
but was awkward because certain objects which were constant in the case $\fre =[-2,2]$ were instead almost periodic.
To overcome this, we found a new approach which, even for $\fre=[-2,2]$, is somewhat simpler than the approach
in \cite{Jost1}.

The idea is to exploit the formula for the diagonal Green's function for $x\in\bbC_+$,
\begin{equation}\lb{7.1}
G_{nn}(x) =\jap{\delta_n, (J-x)^{-1}\delta_n}
\end{equation}
namely (see, e.g., \cite{Rice}),
\begin{equation}\lb{7.2}
G_{nn}(x)=\f{p_{n-1}(x) U_n(x)}{\Wr(x)}
\end{equation}
where $U_n(x)$ is defined by
\begin{equation}\lb{7.3}
U_n(x)=u_n(\zeta) \qquad \x(\zeta)=x \qquad \zeta\in\calF^\intt
\end{equation}
and $\Wr(x)$ is defined by
\begin{equation}\lb{7.4}
\Wr(x)=a_m \bigl(U_{m+1}(x) p_{m-1}(x)-U_m(x) p_m(x)\bigr)
\end{equation}
for $m\geq 1$. The right-hand side is independent of $m$. The funny indices in \eqref{7.4} compared to Wronskians come
from the fact that $U_m$ and $V_m=p_{m-1}$ obey the same difference equation, and RHS of \eqref{7.4} is nothing but
$a_m (U_{m+1}V_m-U_mV_{m+1})$.

In \eqref{7.4}, we can also take $m=0$ if we set $a_0=1$, $p_{-1}(x)=0$, and
\begin{equation}\lb{7.5x}
U_0(x)=u(\zeta;\mu)
\end{equation}
%To see that \eqref{7.4} also holds for $m=0$, we note that
With this choice of $p_{-1}$, $U_0$, and $a_0$,
$U_m$ obeys $a_0 U_0 + b_1 U_1 + a_1 U_2=xU_1$, and similarly for $V_m$. Since $p_{-1}=0$ and $p_0=1$,
\eqref{7.4} for $m=0$ says
\begin{equation}\lb{7.5}
\Wr(x)=-u(\zeta;\mu)
\end{equation}
Here is the key to going from Jost to Szeg\H{o} asymptotics:

\begin{theorem}\lb{T7.1} Suppose $\{a_n,b_n\}_{n=1}^\infty$ obey \eqref{1.13} for some $\{\ti a_n,\ti b_n\}_{n=1}^\infty$ in $\calT_\fre$. Then, uniformly for
$z$ in compact subsets of $\bbC\setminus ([\alpha_1, \beta_{\ell+1}]\cup\sigma(J))$,
\begin{equation}\lb{7.6}
\lim_{n\to\infty} [G_{nn}(z)-\wti G_{nn}(z)]=0
\end{equation}
where $\wti G_{nn}$ is given by \eqref{7.1} with $J$ replaced by $\ti J$.
\end{theorem}

\begin{proof} By the resolvent formula,
\begin{equation}\lb{7.7}
G_{nn}(z)-\wti G_{nn}(z) =\sum_{m,k} G_{nm}(z) (\ti J-J)_{mk}\, \wti G_{kn}(z)
\end{equation}
On compact subsets of $\bbC_+$,
\begin{equation}\lb{7.8}
\abs{G_{kn}(z)}+\abs{\wti G_{kn}(z)} \leq Ce^{-D\abs{k-n}}
\end{equation}
for suitable $C, D>0$. Since $(\ti J-J)_{mk}\to 0$ as $m,k\to\infty$, we get \eqref{7.6} from \eqref{7.7}
and \eqref{7.8}. Using the maximum principle, one extends the result to compact subsets of $\bbC\setminus
([\alpha_1, \beta_{\ell+1}]\cup\sigma(J))$.
\end{proof}

\begin{theorem}\lb{T7.2} Under the hypotheses of Theorem~\ref{T7.1}, uniformly on the same compact subsets
of $\bbC$, we have that
\begin{equation}\lb{7.7x}
\lim_{n\to\infty} \f{G_{nn}(z)}{\wti G_{nn}(z)} =1
\end{equation}
\end{theorem}

\begin{proof} For each fixed $n$, $\wti G_{nn}(z)$ is nonvanishing on the compact subsets under discussion since
neither $\ti u_n$ nor $\ti p_{n-1}$ have zeros there. Since shifting $n$ is equivalent to moving on
the torus, $\wti G_{nn}$ is uniformly bounded away from zero as $n$ varies (cf.\ \eqref{7.10a} below). Therefore,
\eqref{7.6} implies \eqref{7.7x}.
\end{proof}

As a final preliminary on Szeg\H{o} asymptotics, we look at the isospectral torus. If $d\nu\in\calT_\fre$, then
reflection of the Jacobi parameters about $n=0$,
\begin{equation}\lb{7.9}
b_n^{(r)}=b_{-n}, \quad a_n^{(r)}=a_{1-n}, \quad n\in\bbZ
\end{equation}
gives an almost periodic Jacobi matrix in the isospectral torus, so a point we will call $d\nu^{(r)}\in\calT_\fre$.

For $n\in\bbZ$, we denote by $d\nu_n\in \calT_\fre$ the spectral measure of the two-sided Jacobi matrix $\ti J_\nu$
when restricted to $\ell^2(\{n+1,n+2,\dots\})$. In particular, $d\nu_0=d\nu$.

Following paper I, for $x\in\bbC\cup\{\infty\}\setminus\fre$, we define $\z(x)\in{\calF}$ to be the unique point with
$\x(\z(x))=x$, and for $x\in\fre$, we set $\z(x)=\z(x-i0)$.

\begin{theorem}\lb{T7.3} Given $d\nu \in \calT_\fre$, there exist nonvanishing, continuous functions $\alpha(x;\nu)$
and $\beta(x;\nu)$ for $x\in\bbC\setminus [\alpha_1, \beta_{\ell+1}]$ so that the orthonormal polynomials are given by
\begin{equation}\lb{7.10}
p_{n-1}(x;\nu) = \alpha(x;\nu) \f{u(\z(x),\nu^{(r)}_{-n})}{a^{(r)}_{-n} B(\z(x))^{n}}
+
\beta(x;\nu) \f{u(\z(x),\nu_n)}{a_n B(\z(x))^{-n}}
\end{equation}
In particular, $p_{n-1}(x;\nu) B(\z(x))^n$ is asymptotically almost periodic. Moreover, on any compact subset, $K$,
of $\bbC\setminus [\alpha_1, \beta_{\ell+1}]$, there is a constant $C>1$ so that
\begin{equation}\lb{7.10a}
C^{-1} B(\z(x))^n \leq \abs{p_{n-1}(x;\nu)} \leq C B(\z(x))^n
\end{equation}
for all $x\in K$ and $d\nu\in\calT_\fre$.
\end{theorem}

\begin{proof} Define
\begin{equation}\lb{7.11}
u_n^+ (x;\nu) = u_n (\z(x);\nu) \qquad u_n^-(x;\nu) = u_{-n}^+ (x;\nu^{(r)})
\end{equation}
Then $u_n^\pm$ are two solutions of
\begin{equation}\lb{7.12}
a_n v_{n+1} +b_n v_n + a_{n-1} v_{n-1} = xv_n
\end{equation}
and they are linearly independent since one is $L^2$ at $+\infty$ and the other at $-\infty$, and $x$ is not an
eigenvalue of $\ti J_\nu$.

Since $p_{n-1}(x;\nu)$ also solves \eqref{7.12}, we have
\begin{equation}\lb{7.13}
p_{n-1}(x;\nu)=\alpha(x;\nu) u_n^-(x;\nu) + \beta(x;\nu) u_n^+ (x;\nu)
\end{equation}
and Wronskian formulae for $\alpha$ and $\beta$ show that they are real analytic in $\nu\in\calT_\fre$ and analytic
in $x\in\bbC\setminus [\alpha_1, \beta_{\ell+1}]$.

\eqref{7.10} then follows from Theorem~9.2 of paper~I.

Since $\abs{B}<1$ on $\bbD$, the second term multiplied by $B^n$ is exponentially small, and the first is almost
periodic, so $p_{n-1} B^n$ is almost periodic up to an exponentially small error.

The upper bound in \eqref{7.10a} is immediate from \eqref{7.10}, $\abs{B}<1$, and the almost periodicity of $u(z;\nu_n)$.

Since $x$ is not an eigenvalue of $\ti J_\nu$, $\alpha$ is nonvanishing, which proves that for any $K$ and $n\geq N$,
we have a lower bound. Since $p_n$ has no zero in $K$, a lower bound on $n<N$ is immediate. That proves \eqref{7.10a}.
\end{proof}

\begin{theorem}[Szeg\H{o} asymptotics]\lb{T7.4} Let $d\mu\in\Sz(\fre)$ and let $d\ti\mu$ be the measure of the Jacobi matrix
in $\calT_\fre$ for which \eqref{1.13} holds. Then, uniformly on compact subsets of $\bbC\setminus [\alpha_1, \beta_{\ell+1}]$,
\begin{equation}\lb{7.12x}
\f{p_n(x;\mu)}{p_n(x;\ti\mu)} \to \f{u(\z(x);\mu)}{u(\z(x);\ti\mu)}
\end{equation}
In particular, $p_n(x;\mu)B(\z(x))^n$ is asymptotically almost periodic.
\end{theorem}

\begin{remarks} 1. It is not hard to see that the last statement extends to $\bbC\setminus\fre$.

\smallskip
2. In the periodic case, one also has Szeg\H{o} asymptotics in the gaps of $\fre$ except at finitely many points.

\smallskip
3. Since the monic orthogonal polynomials, $P_n(x)$, are related to the orthonormal ones via $P_n(x)=(a_1\cdots a_n) \, p_n(x)$, Szeg\H{o} asymptotics for the monic polynomials immediately follows from \eqref{6.34} and \eqref{7.12x},
\[
\f{P_n(x;\mu)}{P_n(x;\ti\mu)} \to \f{u(\z(x);\mu)/u(0;\mu)}{u(\z(x);\ti\mu)/u(0;\ti\mu)}
\]
\end{remarks}

\begin{proof} It follows from \eqref{7.2} and \eqref{7.5} that
\begin{equation}\lb{7.13x}
\f{p_{n-1} (x;\mu)}{p_{n-1}(x;\ti\mu)} = \f{G_{nn}(x)}{\wti G_{nn}(x)} \,\,
\f{u_n(\z(x);\ti\mu)}{u_n(\z(x);\mu)} \,\, \f{u(\z(x);\mu)}{u(\z(x);\ti\mu)}
\end{equation}
The result is immediate from \eqref{7.7x} and \eqref{6.32} since we can include points below $\alpha_1$
and above $\beta_{\ell+1}$ by the maximum principle and the fact that $p_n(x;\ti\mu)$ is non-vanishing
on $\bbR\setminus [\alpha_1, \beta_{\ell+1}]$.
\end{proof}

%%%%%%%%%%%%%%%%%%%%%%%%%%%%%%%%%%%%%%%%%%%%%%%%%%%%%%%%%%%%%%
\section{$L^2$ Szeg\H{o} Asymptotics on the Spectrum} \lb{s8}
%%%%%%%%%%%%%%%%%%%%%%%%%%%%%%%%%%%%%%%%%%%%%%%%%%%%%%%%%%%%%%

By a standard approximation argument going back to Szeg\H{o}
\cite{Sz20-21}, the function
\[
\int_0^{2\pi} \f{e^{i\theta}+z}{e^{i\theta}-z}
\log(\Ima M(e^{i\theta}))\f{d\theta}{2\pi}
\]
is in $H^2(\bbD)$, so it has nontangential boundary values for a.e.\
$z\in\partial\bbD$.
Since convergent Blaschke products (with a Blaschke condition) are well
known to have boundary values (see
\cite[pp.~249, 310]{Rudin}), $u(z;\mu)$ has boundary values for a.e.\
$z\in\partial\bbD$ and all $d\mu\in\Sz(\fre)$,
and so does $u_n(z;\mu)$ by \eqref{5.9}.

Thus, for Lebesgue a.e.\ $x\in\fre$,
\begin{equation}\lb{8.1}
u_n^+(x;\mu) \equiv u_n(\z(x-i0);\mu)
\end{equation}
exists. Moreover, since $\Ima m(x+i0)\neq 0$ for a.e.\ $x\in\fre$, we
can define a linearly independent solution $u_n^-$ by
\begin{equation}\lb{8.2}
u_n^- (x;\mu) \equiv \ol{u_n^+(x;\mu)}
\end{equation}

This leads to an expansion:
\begin{align}
p_n(x)
&= \f{\Wr(p_{\bddot-1}, u_\bddot^-) u_{n+1}^+(x;\mu) - \Wr(p_{\bddot-1},
u_\bddot^+) u_{n+1}^-(x;\mu)}{\Wr(u_\bddot^+, u_\bddot^-)} \lb{8.3} \\
&= \f{\ol{u_0^+(x;\mu)} u_{n+1}^+(x;\mu) - u_0^+(x;\mu)
\ol{u_{n+1}^+(x;\mu)}}{\Wr(u_\bddot^+, u_\bddot^-)} \lb{8.4}
\end{align}
Given the asymptotics of $u_n^+$ to $\ti u_n^+$, this explains the
expected $L^2$ asymptotic result we'll prove:

\begin{theorem}\lb{T8.1} Let $d\mu\in\Sz(\fre)$ have the form
\eqref{1.6} and let $\ti u_n^+(x)$ be the Jost solution for
the asymptotic point in $\calT_\fre$ {\rm{(}}i.e., the point given by
\eqref{1.13}{\rm{)}}. Then
\begin{equation}\lb{8.5}
\int_\fre\, \biggl| p_n(x) - \f{\Ima (\ol{u(\z(x);\mu)}\, \ti
u_{n+1}^+(x))}{\pi v_\fre(x)}\biggr|^2 w(x)\, dx \to 0
\end{equation}
and
\begin{equation}\lb{8.6}
\int \abs{p_n(x)}^2 \, d\mu_\s(x)\to 0
\end{equation}
where $v_\fre$ is the weight for the reference measure used in \eqref{5.2}.
\end{theorem}

\begin{remarks} 1. $\pi v_\fre(x)$ enters because of the following
calculation:
\begin{align}
\Wr(\ti u_\bddot^+, \ti u_\bddot^-)
&= \ti a_0 (\ti u_1^+\, \ol{\ti u_0^+} - \ol{\ti u_1^+}\, \ti u_0^+)
\lb{8.7} \\
&= -(\ti a_0)^2 \abs{\ti u_0^+}^2 2i \Ima \ti m(x-i0) \lb{8.8} \\
&= 2i \, \f{v_\fre(x)}{\ti w(x)}\, \pi \ti w(x) \lb{8.9} \\
&= 2\pi i \, v_\fre(x) \lb{8.10}
\end{align}
%and the $(2i)^{-1}$ is used to get $\Ima (\cdot)$.
In the above, \eqref{8.8} comes from \eqref{1.21} and \eqref{5.8},
and \eqref{8.9} comes from \eqref{4.14a}, \eqref{5.2} (see
Lemma~\ref{L8.2} below), and \eqref{5.9}, which says that $\ti u_0^+=\ti
a_0^{-1} u(\dott,\ti\mu)$.

\smallskip
2. In case $\fre=[-2,2]$, \eqref{8.5} becomes
\[
\int_{-2}^2\, \biggl| p_n(x) - \f{\Ima (\ol{u(\z(x);\mu)} \,
e^{i(n+1)\theta(x)})}{\sin(\theta(x))}\biggr|^2 w(x)\, dx  \to 0
\]
where $\theta(x)$ is given by $\z(x)=e^{i\theta(x)}$. This is a result
of \cite{PYpams}; see also \cite{Jost1} and
\cite[Sect.~3.7]{Rice}.
\end{remarks}

We define
\begin{align}
k_n^+(x) &= \f{\ol{u(\z(x);\mu)}\, \ti u_{n+1}^+(x)}{2\pi i v_\fre(x)}
\lb{8.11} \\
k_n^-(x) &= \ol{k_n^+(x)} \lb{8.12}
\end{align}
in which case, \eqref{8.5}--\eqref{8.6} become
\begin{equation}\lb{8.13}
\norm{p_n-k_n^+ -k_n^-}_w^2 + \norm{p_n}_\s^2 \to 0
\end{equation}
where $\norm{\cdot}_w$ is the $L^2(\fre,w\,dx)$ norm (we use $\jap{\, ,
\, }_w$ for the inner product) and
$\norm{\cdot}_\s$ is the $L^2(\bbR,d\mu_\s)$ norm. Clearly, \eqref{8.13}
follows from:
\begin{gather}
\norm{p_n}_w^2 + \norm{p_n}_\s^2 =1 \lb{8.14} \\
\norm{k_n^\pm}_w^2 = \tfrac12 \lb{8.15} \\
\lim_{n\to\infty} \jap{k_n^-, k_n^+}_w =0 \lb{8.16} \\
\lim_{n\to\infty} \Real \jap{k_n^-, p_n}_w = \tfrac12 \lb{8.17}
\end{gather}

\eqref{8.14} is the normalization condition on $p_n$, so we only need to
prove \eqref{8.15}--\eqref{8.17}.
We'll need some preliminaries:

\begin{lemma}\lb{L8.2} For a.e. $z\in\partial\bbD$, the boundary value
of $u(z;\mu)$ obeys \begin{equation}\lb{8.18}
\abs{u(z;\mu)}^2 = \f{v_\fre(\x(z))}{w(\x(z))}
\end{equation}
\end{lemma}

\begin{proof} In \eqref{5.2}, $\abs{\prod_{k} B(z,z_k)}$ has $1$ as
boundary value, by standard results on Blaschke
products. By convergence of the Poisson kernel, for a.e.\ $z$ in
$\partial\bbD$, the real part of the exponential
converges to $\log(\f{v_\fre (\x(z))}{w(\x(z))})$.
\end{proof}

\begin{lemma}\lb{L8.3} For any $d\nu\in\calT_\fre$ with weight $w_\nu$,
we have
\begin{equation}\lb{8.19}
\int_\fre \f{dx}{w_\nu(x)} = 2\pi^2 a_0(\nu)^2
\end{equation}
\end{lemma}

\begin{proof} If $\wti G_{00}(z;\nu)$ is the Green's function of the
whole-line Jacobi matrix $\ti J_\nu$ and
$u_n^+(x;\nu)=u_n(\z(x+i0);\nu)$ the boundary value of the Jost
solution, then
\begin{align}
\wti G_{00}(x+i0;\nu)
&= \f{\ol{u_0^+(x;\nu)}\, u_0^+(x;\nu)}{a_0(\nu) [u_1^+(x;\nu)\,
\ol{u_0^+(x;\nu)} - \ol{u_1^+(x;\nu)}\, u_0^+(x;\nu)]} \lb{8.20} \\
&= - \f{1}{a_0(\nu)^2 2i \Ima m(x+i0;\nu)} \lb{8.21} \\
&= \f{i}{2\pi a_0(\nu)^2  w_\nu(x)} \lb{8.22}
\end{align}
so
\begin{equation}\lb{8.23}
\f{1}{\pi}\, \Ima \wti G_{00}(x+i0;\nu) = \f{1}{2\pi^2 a_0(\nu)^2 w_\nu(x)}
\end{equation}

But the whole-line Jacobi matrix $\ti J_\nu$ has purely a.c.\ spectrum
$\sigma(\ti J_\nu)=\fre$ and the density of the probability
spectral measure for $\ti J_\nu$ and $\delta_0$ is $\f{1}{\pi} \Ima \wti
G_{00}(x+i0;\nu)$, so
\begin{equation}\lb{8.24}
\f{1}{\pi} \int_\fre \Ima\wti G_{00}(x+i0;\nu)\, dx = 1
\end{equation}
\eqref{8.23} and \eqref{8.24} imply \eqref{8.19}.
\end{proof}

\begin{proposition}\lb{P8.4} \eqref{8.15} holds.
\end{proposition}

\begin{proof} By \eqref{5.9} and \eqref{8.11},
\begin{equation}\lb{8.25}
\abs{k_n^+(x)}^2 = \f{\abs{u(\z(x);\mu)}^2
\abs{u(\z(x);\ti\mu_{n+1})}^2}{4\pi^2 (\ti a_{n+1})^2 v_\fre(x)^2}
\end{equation}
so, by Lemma~\ref{L8.2},
\begin{equation}\lb{8.26}
\abs{k_n^+(x)}^2 = \f{1}{4\pi^2 (\ti a_{n+1})^2 w(x) \ti w_{n+1}(x)}
\end{equation}
and so,
\begin{equation}\lb{8.27}
\int_\fre\, \abs{k_n^+(x)}^2 w(x)\, dx = \f{1}{4\pi^2(\ti a_{n+1})^2}
\int_\fre \f{dx}{\ti w_{n+1}(x)} = \f12
\end{equation}
by Lemma~\ref{L8.3}. Since $\abs{k_n^-}=\abs{k_n^+}$, we get the same
result for $\norm{k_n^-}_w^2$.
\end{proof}

\begin{lemma}\lb{L8.5} Let $f\in L^1(\fre,d\rho_\fre)$. Then
\begin{equation}\lb{8.28}
\lim_{n\to\infty} \int_\fre B(\z(x))^{n} f(x)\, d\rho_e(x) =0
\end{equation}
Moreover, \eqref{8.28} holds uniformly on norm compact subsets of
$L^1(\fre,d\rho_\fre)$.
\end{lemma}

\begin{proof}
Without loss of generality, assume that $f$ is real-valued. Then by
Corollary 4.6 of paper I, we obtain
\begin{align} \lb{8.29}
\int_\fre B(\z(x))^{n} f(x)\, d\rho_e(x) = \int_0^{2\pi}
B(e^{i\theta})^{n}f(\x(e^{i\theta}))\,\f{d\theta}{2\pi}
\end{align}
By the Cauchy theorem, $\{B^n\}_{n\in\bbZ}$ forms an orthonormal system
in $L^2(\partial\bbD,\f{d\theta}{2\pi})$. Hence it follows from the
Bessel inequality that RHS of \eqref{8.29} converges to zero for any
$L^2$-function. The general case of $L^1$-functions and the result on
uniform convergence on norm compacts follow by approximation.
\end{proof}

\begin{remark} The above result can be also established via a stationary
phase argument.
\end{remark}

\begin{proposition} \lb{P8.6} \eqref{8.16} holds.
\end{proposition}

\begin{proof} By the same calculation that was used in the proof of
Proposition~\ref{P8.4},
\begin{equation}\lb{8.30}
\jap{k_n^-, k_n^+}_w = \int_\fre f_n(x) B^{2n+2}(\z(x))\, dx
\end{equation}
where
\begin{equation}\lb{8.31}
f_n(x) = -\f{1}{4\pi^2 (\ti a_{n+1})^2} \,
\f{u(\z(x);\ti\mu_{n+1})^2}{v_\fre(x)} \,
\f{\abs{u(\z(x); \mu)}^2}{u(\z(x); \mu)^2}
\end{equation}

For $d\nu\in\calT_\fre$, let
\begin{equation}\lb{8.32}
f(x;\nu) = -\f{1}{4\pi^2 a_0(\nu)^2}\, \f{u(\z(x);\nu)^2}{v_\fre(x)} \,
\f{\abs{u(\z(x);\mu)}^2}{u(\z(x);\mu)^2}
\end{equation}
By Lemma~\ref{L8.2}, the $f$'s are all in $L^1$ (with $L^1$ norm $1/2$
by Lemma \ref{L8.3}) and $f$ is $L^1$ continuous in $\nu$. So,
since $\calT_\fre$ is compact, we see from Lemma~\ref{L8.5} that the
integral in \eqref{8.30} goes to zero.
\end{proof}

This leaves \eqref{8.17}. The argument is somewhat complicated in case
there are bound states, especially if there are infinitely many.
So let us consider it first when $d\mu$ has no point masses in
$\bbR\setminus\fre$.

\begin{proposition}\lb{P8.10} Suppose $d\mu$ has support $\fre$ so that
$u(z;\mu)$ is nonvanishing on $\bbD$. Then
\eqref{8.17} holds.
\end{proposition}

\begin{proof} We claim that
\begin{equation}\lb{8.34}
\begin{split}
\Real \biggl[ \int_\fre & \ol{k_n^-(x)}\, p_n(x) w(x)\, dx \biggr] \\
&= \f12 \int_{\partial\calF\cap\partial\bbD} \f{\ol{u(z;\mu)}\, \ti
u_{n+1}(z)}{2\pi i v_\fre(\x(z))}\,
p_n(\x(z))\, w(\x(z)) \x'(z)\, dz
\end{split}
\end{equation}
where the integral is evaluated counterclockwise. As
$\Real\ol{k_n^-}=\f12 k_n^+ + \f12 k_n^-$ and $\Real p_n(x)=p_n(x)$,
the $k_n^+$ term directly gives the counterclockwise integral over
$\bbC_+\cap\partial\calF\cap\partial\bbD$ (since
$\x'(z)$ is positive there). Since $u$ and $\ti u_{n+1}^+$ are real on
$\bbR$,
and $\x'$ and $i$ flip signs under $e^{i\theta}\to e^{-i\theta}$, the
$k_n^-$ term gives the integral over
$\partial\calF\cap\partial\bbD\cap\bbC_-$.

Notice next that, by \eqref{8.18},
\begin{equation}\lb{8.35}
\ol{u(z;\mu)}\, \f{w(\x(z))}{v_\fre(\x(z))} = \f{1}{u(z;\mu)}
\end{equation}
so
\begin{equation}\lb{8.36}
\text{LHS of \eqref{8.34}} = \f{1}{4\pi i}
\int_{\partial\calF\cap\partial\bbD}
\f{\ti u_{n+1}(z) p_n(\x(z))}{u(z;\mu)}\, \x'(z)\, dz
\end{equation}

By \eqref{6.28}, \eqref{5.9}, and the choice of $d\ti\mu$, the integrand
in \eqref{8.36}, call it $F$, is automorphic under
$\Gamma$. Since $F$ is real on $\bbR$, we have $F(\bar z)=\ol{F(z)}$.
Moreover, there are $\gamma\in\Gamma$ so that for $z\in
C_\ell^+$, we have $\ol{\gamma(z)}=z$, so we conclude that $F$ is real
on $C_\ell^+$ and $C_\ell^-$. Thus, orienting the
contours counterclockwise about $0$, we get
\[
\int_{C_\ell^+ \cup C_\ell^-} F(z)\, dz =0
\]
since $C_\ell^+$ and $C_\ell^-$
run in opposite directions. It follows that
\begin{equation}\lb{8.37}
\text{LHS of \eqref{8.34}} = \f{1}{4\pi i} \int_{\partial\calF}
\f{\ti u_{n+1}(z) p_n(\x(z))}{u(z;\mu)}\, \x'(z)\, dz
\end{equation}

Inside $\calF$, the integrand is regular except at $z=0$. Since $p_n$ is
a polynomial of degree $n$ in $\x(z)$,
and $\x(z)$ has a simple pole at $z=0$, $z^n p_n(\x(z))$ is regular at
$z=0$. By \eqref{5.9}, $\ti u_{n+1}(z)/B(z)^{n+1}$
is regular at $z=0$. Thus, $\ti u_{n+1}(z) p_n(\x(z))$ has a first-order
zero at $z=0$. $u(z)$ is regular there and
$\x'(z)$ has a double pole. So the integrand in \eqref{8.37} has a
simple pole at $z=0$ and we conclude that
\begin{equation}\lb{8.38}
\text{LHS of \eqref{8.34}} = \f12
\left[\f{u_{n+1}(z;\mu)p_n(\x(z))}{zu(z;\mu)}\bigg|_{z=0}\right]\,
\f{u_{n+1}(0;\ti\mu)}{u_{n+1}(0;\mu)}\, [z^2\x'(z)|_{z=0}]
\end{equation}
The first factor in \eqref{8.38} is $\left. z^{-1}
G_{n+1,n+1}(\x(z))\right|_{z=0}$, which is
\begin{equation}\lb{8.39}
\lim_{z\to 0} z^{-1} \biggl( -\f{1}{\x(z)} + O\biggl(
\f{1}{\x(z)^2}\biggr)\biggr) = -\f{1}{x_\infty}
\end{equation}
The third factor is
\begin{equation}\lb{8.40}
\lim_{z\to 0} z^2 \biggl( -\f{x_\infty}{z^2} + O(1)\biggr) =-x_\infty
\end{equation}
so
\begin{equation}\lb{8.41}
\text{LHS of \eqref{8.34}} = \f12\,
\f{u_{n+1}(0;\ti\mu)}{u_{n+1}(0;\mu)} \to \f12
\end{equation}
by Theorem~\ref{T6.6}.
\end{proof}

\begin{proposition} \lb{P8.8} If $d\mu$ has support $\fre$ plus finitely
many mass points in $\bbR\setminus\fre$,
then \eqref{8.17} holds.
\end{proposition}

\begin{proof} We follow the proof of the last proposition until we get
to \eqref{8.36}. However, $u$ now has a
pole at each $z_k$ in $\calF$ with
\begin{equation}\lb{8.42}
\x(z_k) = x_k\in\sigma(J)
\end{equation}
Thus, the integrand can have poles (but only finitely many) in
$\calF^\intt$ and also on $C_j^\pm$. Interpret
\eqref{8.37} as taking principal parts at the poles on $C_j^\pm$. Each
such pole contributes with half of $2\pi i$ times the residue,
so we get $2\pi i$ times the residue if we only count the poles in
$\calF$ (i.e., in $C_j^+$ but not in $C_j^-$).

The residue at $z_k$ is
\begin{equation}\lb{8.43}
\f{B(z_k)^{n+1} u(z_k;\ti\mu_{n+1}) p_n(x_k) \x'(z_k)}{2 \ti a_{n+1}
u'(z_k;\mu)}
\end{equation}
As $\sum_n \abs{p_n(x_k)}^2 = 1/\mu(\{x_k\})$, $\abs{B(z_k)} <1$ and
$\sup_n \abs{u(z_k;\ti\mu_{n+1})} <\infty$, the quantity in \eqref{8.43}
goes to zero. Since there are finitely many of these poles, their
contribution vanishes in the limit and LHS of \eqref{8.34} converges to
$1/2$.
\end{proof}

Finally, we turn to the general case. The following completes the proof
of Theorem~\ref{T8.1}:

\begin{proposition} \lb{P8.9} For any $d\mu\in\Sz(\fre)$, \eqref{8.17}
holds.
\end{proposition}

\begin{proof} Following Peherstorfer--Yuditskii \cite{PYpams}, we'll
approximate $u$ by one with a finite number of
zeros, but to preserve the fact that we need certain functions to be
automorphic, we also modify $\ti u_n$.

Label all the point masses of $d\mu$ in a single sequence
$\{x_k\}_{k=1}^\infty$ with corresponding points
$z_k\in\calF$ such that $\x(z_k)=x_k$. Let
\begin{equation}\lb{8.44}
u^{(m)}(z;\mu) =\prod_{k=1}^m B(z,z_k) \veps(z;\mu)
\end{equation}
and denote by $d\ti\mu^{(m)}$ the measure in the isospectral torus whose
Jost function has the same character as $u^{(m)}$. Define
\begin{equation}\lb{8.45}
k_n^{(m)+}(x) = \f{\ol{u^{(m)}(\z(x);\mu)} \,
u_{n+1}^+(x;\ti\mu^{(m)})}{2i \pi v_\fre (x)}
\end{equation}

Clearly, it suffices to prove that
\begin{equation}\lb{8.46}
\lim_{m\to\infty} \norm{k_n^{(m)+} - k_n^+}_w \to 0
\end{equation}
uniformly in $n$, and that
\begin{equation}\lb{8.47}
\lim_{m\to\infty}\, \lim_{n\to\infty} \abs{\Real\jap{k_n^{(m)+},
p_n}-\tfrac12}=0
\end{equation}

Since $\prod_{k=1}^m B(z,z_k)\to\prod_{k=1}^\infty B(z,z_k)$ uniformly
on compacts, the characters converge. Moreover,
this convergence of $B$'s is pointwise on $\partial\bbD$. The first
implies convergence of $u(\z(x);\ti\mu_{n+1}^{(m)})$
to $u(\z(x);\ti\mu_{n+1})$ away from the band edges (uniformly in $n$
and $x$  as $m\to\infty$) with uniform square root
bounds. This plus \eqref{8.26} yields \eqref{8.46}.

The proof of \eqref{8.47} follows the proof of Proposition~\ref{P8.8}.
The fact that we've arranged for the functions to
be automorphic allows the cancellation of the $C_j^+$ and $C_j^-$
integrals, and since there are only finitely many poles
away from $z=0$, we get convergence in \eqref{8.43} and hence in
\eqref{8.47}.
\end{proof}

\bigskip

%%%%%%%%%%%%%%%%%%%%%%%%%%%%%%%

\end{document}